\documentclass{amsart}
\usepackage{amssymb}
\usepackage{amsthm}
\usepackage{tikz}
\usepackage{hyperref}

\newtheorem{conjecture}{Conjecture}
\newtheorem{theorem}{Theorem}[section]
\newtheorem{remark}[theorem]{Remark}

\newtheorem{lemma}[theorem]{Lemma}

\newtheorem{proposition}[theorem]{Proposition}
\newtheorem{corollary}[theorem]{Corollary}
\newtheorem{problem}{Problem}

\newcommand{\BB}{\mathbb B}

\def\bC{\mathbb C}

\def\bT{\mathbb T}

\begin{document}
\title{Exceptional Zeros of $L$-series and Bernoulli-Carlitz Numbers}

\author{Bruno Angl\`es  \and Tuan Ngo Dac \and Floric Tavares Ribeiro}

\address{
Universit\'e de Caen Normandie, 
Laboratoire de Math\'ematiques Nicolas Oresme, 
CNRS UMR 6139, 
Campus II, Boulevard Mar\'echal Juin, 
B.P. 5186, 
14032 Caen Cedex, France.
}
\email{bruno.angles@unicaen.fr, tuan.ngodac@unicaen.fr, floric.tavares-ribeiro@unicaen.fr}

\address{Department of Mathematics, The Ohio State University, 231 W
. 18th Ave., Columbus, Ohio 43210}
\email{goss@math.ohio-state.edu}
\address{ 
Institut Camille Jordan, UMR 5208
Site de Saint-Etienne,
23 rue du Dr. P. Michelon,
42023 Saint-Etienne, France
}
\email{federico.pellarin@univ-st-etienne.fr}

\begin{abstract} Bernoulli-Carlitz numbers were introduced by L. Carlitz in 1935, they are the analogues in positive characteristic of Bernoulli numbers. We prove a conjecture formulated  by F. Pellarin and the first author on the non-vanishing modulo a given prime of families of Bernoulli-Carlitz numbers. We then show that the "exceptional zeros" of certain $L$-series are intimately connected to the Bernoulli-Carlitz numbers.  \end{abstract}
\date{ \today}
\maketitle
\centerline{With an appendix by B. Angl\`es, D. Goss, F. Pellarin, F. Tavares Ribeiro}
\tableofcontents
\section{Introduction}${}$\par
Recently, M. Kaneko and D. Zagier have introduced the $\mathbb Q$-algebra of finite multiple zeta values which is a sub-$\mathbb Q$-algebra of $\mathcal A:=\frac{\prod_{p} \mathbb F_p}{\oplus_{p}\mathbb F_p}$ ($p$ runs through the prime numbers). This algebra of finite multiple zeta values contains the following elements:
$$\forall k \geq 2, {\mathcal Z}(k)=\left(\left(\frac{B_{p-k}}{k}\right)_p\right)\in \mathcal A,$$
where $B_n$ denotes the $n$th Bernoulli number. It is not known that the algebra of finite multiple zeta values is non-trivial. In particular, it is an open problem to prove that ${\mathcal Z}(k)\not =0$ for $k\geq 3, k\equiv 1\pmod{2}$ (observe that ${\mathcal Z}(k)=0$ if $k\geq 2, k\equiv 0\pmod{2}$). This latter problem is equivalent to the following:
\begin{conjecture}\label{ConjectureKAZ} Let $k\geq 3$ be an odd integer. Then, there exist infinitely many primes $p$ such that $B_{p-k}\not \equiv 0\pmod{p}.$
\end{conjecture}
Let $k\geq 3$ be an odd integer. M. Kaneko (\cite{KAZ}) remarked that, viewing the $B_{p-k}$'s as being random modulo $p$ when $p$ varies through the prime numbers, taking into account that $\sum_p\frac{1}{p}$ diverges, then it is reasonable to expect that there exist infinitely many prime numbers $p$ such that $B_{p-k}\equiv 0\pmod{p}.$\par
Let $\mathbb F_q$ be a finite field having $q$ elements, $q$ being a power of a prime number $p$,  and  let $\theta$ be an indeterminate over $\mathbb F_q.$ In 1935, L. Carlitz has introduced the analogues of Bernoulli numbers for $A:=\mathbb F_q[\theta]$ (\cite{CAR}). The Bernoulli-Carlitz numbers, $BC_n\in K:=\mathbb F_q(\theta),$ $n\in \mathbb N,$  are defined as follows:\par
\noindent - $ BC_n=0$ if $n\not \equiv 0\pmod{q-1},$\par
\noindent - for $n\equiv 0\pmod{q-1},$ we have:
$$\frac{BC_n}{\Pi(n)}=\frac{\zeta_A(n)}{\widetilde{\pi}^n},$$
where $\Pi(n)\in A$ is the Carlitz factorial (\cite{GOS}, chapter 9, paragraph 9.1), $\widetilde{\pi}$ is the Carlitz period (\cite{GOS}, chapter 3, paragraph 3.2), and $\zeta_A(n):=\sum_{a\in A, a{\rm \, monic}}\frac{1}{a^n}\in K_\infty:=\mathbb F_p((\frac{1}{\theta}))$ is the value at $n$ of the Carlitz-Goss zeta function.  The Bernoulli-Carlitz numbers are  connected to Taelman's class modules introduced in \cite{TAE1} (see for example \cite{TAE2} and \cite{AT}). L. Carlitz established a von-Staudt result for these numbers (\cite{GOS}, chapter 9, paragraph 9.2), and as an easy consequence, we get that if  $P$ is a monic irreducible polynomial in $A,$ then $BC_n$ is $P$-integral for $0\leq n\leq q^{\deg_\theta P}-2.$  It is natural to ask if  Conjecture \ref{ConjectureKAZ} is valid in the carlitzian context.  In this paper, we prove a stronger result which answers positively to a Conjecture formulated in \cite{AP}:
\begin{theorem}\label{TheoremATR-BC} Let $N\geq 2$ be an integer, $N\equiv 1\pmod{q-1}.$ Let $\ell_q(N)$ be the sum of the digits in base $q$ of $N.$ Let $P\in A$ be a monic irreducible polynomial of degree $d$ such that $q^d>N.$ If $d\geq \frac{\ell_q(N)-1}{q-1}N,$ then:
$$BC_{q^d-N}\not \equiv 0\pmod{P}.$$
\end{theorem}
The above Theorem is linked with the study of exceptional  zeros of certain $L$-series introduced in 2012 by F. Pellarin (\cite{PEL}), but from a slightly different point of view. More precisely, let $N$ be as above and for simplicity we assume that $\ell_q(N)\geq q,$ let $t$ be an indeterminate over $K_\infty,$ let's consider:
$$\mathcal L_N(t):=\sum_{d\geq 0}\sum_{a\in A_{+,d}} \frac{a^N}{a(t)}\in  A[[\frac{1}{t}]]^\times,$$
where $A_{+,d}$ is the set of monic elements in $A$ of degree $d.$  It was already noticed by F. Pellarin (\cite{PEL2})  that such $L$-series can be related with Anderson's solitons and should play an important role in the arithmetic theory of function fields. Let $\mathbb C_\infty$ be the completion of a fixed algebraic closure of $K_\infty.$ Then, one can show that  $\mathcal L_N(t)$ converges on $\{ x\in \mathbb C_\infty, v_\infty(x)<0\},$ where $v_\infty$ is the valuation on $\mathbb C_\infty$ normalized such that $v_\infty(\theta)=-1.$   Furthermore, one can easily see that the elements of $S:=\{ \theta^{q^j}, j\in \mathbb Z, q^j\leq N \}$ are zeros of the function $\mathcal L_N(t).$ We call  the zeros of $\mathcal L_N(t)$ which  belong to $\{ x\in \mathbb C_\infty, v_\infty(x)<0\}\setminus S$  the \emph{exceptional zeros} of $\mathcal L_N(t).$ Let's briefly describe the case $q=p$. In this case, the exceptional zeros of $\mathcal L_N(t)$ are simple, belong to $\mathbb F_p((\frac{1}{\theta}))$ and  are the eigenvalues of a certain $K$-linear endomorphism $\phi^{(N)}_t$ of a finite dimensional $K$-vector space $H(\phi^{(N)})$ connected to the generalization of Taelman's class modules introduced in \cite{APT}.  The proof of the fact that the exceptional zeros are simple and "real" uses combinatorial techniques introduced by F. Diaz-Vargas (\cite{DIV}) and J. Sheats (\cite{SHE}). Furthermore, if $p^d>N,$ then:
$$BC_{p^d-N}\frac{(-1)^{\frac{\ell_p(N)-p}{p-1}}\, \prod_{l=0}^k \prod_{n=0, n\not =l}^{d-1} (\theta^{p^l}-\theta^{p^n})^{n_l}}{\Pi(N)\Pi(p^d-N)}=\det_K\, (\theta^{p^d}{\rm Id} -\phi^{(N)}_t\mid_{H(\phi^{(N)})}),$$
where $\Pi(.)$ is the Carlitz factorial, and $N=\sum_{l=0}^k n_l p^l, n_l\in \{0, \ldots, p-1\}.$ Since the eigenvalues of $\phi^{(N)}_t$ are exactly in this situation the exceptional zeros of $\mathcal L_N(t),$ we also obtain another proof of Theorem \ref{TheoremATR-BC} as  a consequence of the fact that: $$\det_{K[Z]}\, (Z{\rm Id} -\phi^{(N)}_t\mid_{H(\phi^{(N)})})\in \mathbb F_p[Z, \theta],$$ and therefore ($P$ is a monic irreducible polynomial of degree $d$):
$$\det_K\, (\theta^{p^d}{\rm Id} -\phi^{(N)}_t\mid_{H(\phi^{(N)})})\equiv \det_K\, (\theta{\rm Id} -\phi^{(N)}_t\mid_{H(\phi^{(N)})})\pmod{P}.$$
Let's observe that Theorem \ref{TheoremATR-BC} implies the following (see \cite{AP}, page 248):
$$\sum_{d\geq 0}\sum_{a\in A_{+,d}}\frac{(a')^N}{a}\not =0,$$
where $a'$ denotes the derivative of $a$ and $N\equiv 1\pmod{q-1}.$ In the appendix of this paper, we discuss a digit principle for such Euler type  sums.\par
We mention that the construction of Kaneko-Zagier's objects in the positive characteristic world is the subject of a forthcoming work of F. Pellarin and R. Perkins (\cite{PP}), they prove, in this context,  that the algebra of finite multiple zeta values is non-trivial. In this situation, it would be very interesting to examine the validity of Conjecture \ref{ConjectureKAZ} for Bernoulli-Goss numbers (see \cite{AO} for a special case).\par
%%%%%%%%%%%%%%%%%%%%%%%%%%%%%%%%%%%%%%%%%%%%%%%%%%%%%%%%%%%%%%%%%%%%%%%%%
\section{Proof of Theorem \ref{TheoremATR-BC}}
\subsection{Notation}
Let $\mathbb F_q$ be a finite field having $q$ elements and let $p$ be the characteristic of $\mathbb F_q.$ Let $\theta $ be an indeterminate over $\mathbb F_q$ and let $A= \mathbb F_q[\theta], K=\mathbb F_q(\theta), K_\infty=\mathbb F_q((\frac{1}{\theta})).$ Let $\mathbb C_\infty$ be the completion of a fixed algebraic closure of $K_\infty$.  Let $v_\infty: \mathbb C_\infty \rightarrow \mathbb Q\cup\{+\infty\}$ be the valuation on $\mathbb C_\infty$ normalized such that $v_\infty(\theta)=-1.$ For $d\in \mathbb N, $ let $A_{+,d}$ be the set of monic elements in  $A$ of degree $d.$\par

\subsection{The $L$-series $L_N(t)$}${}$\par

Let $N\geq 1$ be an integer. Let $t$ be an indeterminate over $\mathbb C_\infty.$ Let $\mathbb T_t$ be the Tate algebra in the variable $t$ with coefficients in $\mathbb C_\infty.$  Let's set:
$$L_N(t)=\sum_{d\geq 0}\sum_{a\in A_{+,d}}\frac{a(t)^N}{a}\in \mathbb T_t^\times.$$
Then, we can write:
$$L_N(t)=\sum_{i\geq 0} \alpha_{i,N}(t) \theta^{-i}, \, \alpha_{i,N}(t)\in \mathbb F_q[t].$$
Observe that $\alpha_{0,N}(t)=1.$ 
\begin{lemma}\label{LemmaATR1} We have:
$$\forall i\geq 0, \deg_t\alpha_{i,N}(t) \leq N({\rm Max}\{\frac{\log(i)}{\log(q)}, 0\}+[\frac{\ell_q(N)}{q-1}]+1).$$
In particular $L_N(t)$  is an entire function.
\end{lemma}
\begin{proof}  Let $u=[\frac{\ell_q(N)}{q-1}]\in  \mathbb N.$ This Lemma is a consequence of the proof of \cite{AP}, Lemma 7. We give a proof for the convenience of the reader. We will use the following elementary fact (\cite{AP}, Lemma 4):\par
\noindent Let $s\geq 1$ be an integer and let $t_1, \ldots, t_s$ be $s$ indeterminate over $\mathbb F_p.$ If $d(q-1)>s$ then $\sum_{a\in A_{+,d}} a(t_1)\cdots a(t_s)=0.$\par  
\noindent If $a$ is a monic polynomial in $A,$ we will set:
$$<a>_\infty=\frac{a}{\theta^{\deg_\theta a}}\in 1+\frac{1}{\theta}\mathbb F_q[[\frac{1}{\theta}]].$$
Let $S_d:=\sum_{a\in A_{+,d}}\frac{a(t)^N}{a}.$ Observe that:
$$\deg_tS_d= dN.$$ We have :
$$S_d=\frac{1}{\theta^d} \sum_{a\in A_{+,d}}a(t)^N<a>_\infty^{-1}.$$
Observe that ($p$-adically) $-1=\sum_{n\geq 0} (q-1)q^{n}.$ For $m\geq 0,$ set:
$$y_m=\sum_{n=0}^{m} (q-1) q^{n}.$$
Then:
$$y_m\equiv -1\pmod{q^{m+1}}, \ \ \ell_q(y_m)= (m+1)(q-1).$$
Therefore:
$$v_\infty(\sum_{a\in A_{+,d}}a(t)^N<a>_\infty^{-1}-\sum_{a\in A_{+,d}}a(t)^N<a>_\infty^{y_m})\geq q^{m+1},$$
where $v_\infty$ is the $\infty$ -adic valuation on $\mathbb T_t$ such that $v_\infty (\theta)=-1.$
Thus, if $\ell_q(N)+ (m+1)(q-1)<d(q-1),$ we get:
$$\sum_{a\in A_{+,d}}a(t)^Na^{y_m}=0.$$
 We therefore get, if $d\geq u+2$ :
$$v_\infty(S_d)\geq d+ q^{d-u-1}.$$
This implies that $L_N(t)$ is an entire function.
Let $j$ such that $t^j$ appears in $\alpha_{i,N}(t).$ Let $x=[\frac{j}{N}].$ Let $d$ be minimal such that $t^j$ comes from $S_d.$ We must have  $d\geq x$ and $j\leq dN.$ Furthermore, if $d\geq u+2,$ we have:
$$i\geq d+ q^{d-u-1}.$$
In particular:
$$i\geq d+q^{d-u-1}\geq q^{d-u-1}.$$
Therefore:
$$d\leq {\rm Max}\{\frac{\log(i)}{\log(q)}, 0\}+u+1.$$
\end{proof}

%%%%%%%%%%%
\subsection{The two variable polynomial $B_N(t, \theta)$}\label{BN}${}$\par
Let $N\geq 2, $ $N\equiv 1\pmod{q-1}.$ If $\ell_q(N)=1,$ we set $B_N(t,\theta)=1.$ Let's assume that $\ell_q(N)\not =1$ and let's set $s=\ell_q(N) \geq 2.$ Let $t_1, \ldots, t_s$ be $s$ indeterminates over $\mathbb C_\infty.$ Let $\mathbb T_s$ be the Tate algebra in the indeterminates $t_1, \ldots, t_s$ with coefficients in $\mathbb C_\infty.$ Let $\tau: \mathbb T_s\rightarrow \mathbb T_s$ be the continuous morphism of $\mathbb F_q[t_1, \ldots, t_s]$-algebras such that $\forall c\in \mathbb C_\infty, \tau (c)=c^q.$  For $i=1, \ldots , s,$ we set:
$$\omega(t_i)=\lambda_\theta\prod_{j\geq 0} (1-\frac{t_i}{\theta^{q^j}})^{-1} \in \mathbb T_s,$$
where $\lambda_\theta$ is a fixed $(q-1)$th-root of $-\theta $ in $\mathbb C_\infty.$ Set:
$$\widetilde{\pi} = \lambda_\theta \theta \prod_{j\geq 1} (1-\theta^{1-q^j})^{-1}.$$
Set:
$$L_s=\sum_{d\geq 0} \sum_{a\in A_{+,d}}\frac{a(t_1)\cdots a(t_s)}{a} \in \mathbb T_s^\times.$$
We also set:
$$\mathbb B_s=(-1) ^{\frac{s-1}{q-1}} \frac{ L_s \omega(t_1)\ldots \omega(t_s)}{\widetilde{\pi}}\in \mathbb T_s.$$
Then, by \cite{APT}, Lemma 7.6 (see also \cite{AP}, Corollary 21), $\mathbb B_s\in \mathbb F_q[t_1, \ldots, t_s, \theta]$ is a monic polynomial in $\theta$ of degree $r=\frac{s-q}{q-1}.$ Write $N=\sum_{n=1}^{\ell_q(N)} q^{e_n}, \, e_1\leq e_2\leq \cdots \leq e_{\ell_q(N)}.$
We set:
$$B_N(t, \theta) = \mathbb B_s\mid_{t_i=t^{q^{e_i}}}\in \mathbb F_q[t, \theta].$$
We observe that $B_N(t, \theta)$ is a monic polynomial in $\theta$ such that $\deg_\theta B_N(t, \theta)=r.$\par
\begin{lemma}\label{LemmaATR5}
Let $N\geq 2,$ $N\equiv 1\pmod{q-1}.$ Then:\par
\noindent 1) $B_N(t^p, \theta^p)= B_N(t, \theta)^p.$\par
\noindent 2) $B_{qN}(t, \theta)=B_N(t^q, \theta).$\par
\noindent 3) We have:
$$B_N(t, \theta)\equiv (\theta-t)^r-r(t^q-t)(\theta-t)^{r-1} \pmod{(t^q-t)^2 \mathbb F_p[t, \theta]}.$$
\noindent 4) If $N\equiv 0\pmod{p},$ then $B_N(t, \theta)\in \mathbb F_p[t^p, \theta].$
\end{lemma}
\begin{proof} Recall that:
$$ B_N(t, \theta) =\frac{(-1)^{\frac{\ell_q(N)-1}{q-1}}}{\widetilde{\pi}}\sum_{d\geq 0} \sum_{a\in A_{+,d}} \frac{a(t)^N}{a}\, \prod_{l=0}^k\omega (t^{q^l})^{n_l},$$
where $N=\sum_{l=0}^k n_l q^l,$ $n_0, \ldots, n_k\in \{ 0, \ldots q-1\}.$ Observe that:
$$\sum_{d\geq 0} \sum_{a\in A_{+,d}} \frac{a(t)^N}{a} \in \mathbb F_p[t][[\frac{1}{\theta}]].$$
Thus:
$$B_N(t, \theta)\in \mathbb F_p[t, \theta].$$
Thus we get assertion 1). Assertion 2) is a consequence of the definition of $B_N(t, \theta).$ Let $\zeta \in \mathbb F_q.$ By \cite{AP2}, theorem 2.9, we have:
$$\omega(t)\mid_{t=\zeta}= \exp_C(\frac{\widetilde{\pi}}{\theta-\zeta}),$$
where $\exp_C: \mathbb C_\infty \rightarrow \mathbb C_\infty$ is the Carlitz exponential (\cite{GOS}, chapter 3, paragraph 3.2). Now, by \cite{PEL} Theorem 1, we get:
$$\sum_{d\geq 0} \sum_{a\in A_{+,d}} \frac{a(t)^N}{a}\mid_{t=\zeta}=\sum_{d\geq 0} \sum_{a\in A_{+,d}} \frac{a(\zeta)}{a}= \frac{\widetilde {\pi}} {(\theta-\zeta) \exp_C(\frac{\widetilde{\pi}}{\theta-\zeta})}.$$
But observe that:
$$\exp_C(\frac{\widetilde{\pi}}{\theta-\zeta})^{q-1}= -(\theta-\zeta).$$
Thus:
$$B_N(t, \theta)\mid_{t=\zeta}= (\theta-\zeta)^r.$$
Observe that $\frac{d}{dt} B_N(t, \theta) $ is equal to:\par
$$\frac{n_0(-1)^{r+1}}{\widetilde{\pi}}\left(\prod_{l=0}^k\omega (t^{q^l})^{-n_l}\right)\sum_{d\geq 0} \sum_{a\in A_{+,d}}\left( \frac{\frac{d}{dt} (a(t)) a(t)^{N-1}}{a}-\frac{a(t)^N}{a}\, \frac{\frac{d}{dt} (\omega (t))}{ \omega(t)}\right).$$
Thus we get  assertion 4).
\noindent Since for $\ell_q(N)=q,$ we have $B_N(t)=1.$ We get:
$$\forall \zeta\in \mathbb F_q, \frac{d}{dt} B_N(t, \theta) \mid_{t=\zeta}=0.$$
This concludes the proof of the Lemma.
\noindent 
\end{proof}
\begin{lemma}\label{LemmaATR6} Let $N\geq 2, N\equiv 1\pmod{q-1}.$ Then $\deg_tB_N(t, \theta)\geq p$  if $r\geq 1$ and the total degree in $t, \theta$ of $B_N(t, \theta)$ is less than or equal to  ${\rm Max}\{ rN+r-2, 0\}.$ Furthermore $B_N(t, \theta)$ (as a polynomial in $t$) is a primitive polynomial.
\end{lemma}
\begin{proof} Recall that if $r=0$ then $B_N(t, \theta)=1.$ Let's assume that $r\geq 1.$  Observe that by Lemma \ref{LemmaATR5},we have:
$$ B_N(t, 0) \equiv -(-t)^{r-1}(t+r(t^q-t))\pmod{(t^q-t)^2\mathbb F_p[t]}.$$
In particular $\deg_tB_N(t, \theta)\geq p.$  Let $x\in \mathbb C_\infty$ such that $v_\infty(x)>\frac{-1}{N}.$ Then:
$$\sum_{d\geq}\sum_{a\in A_{+,d} }\frac{a(x)^N}{a} =\prod_{P{\rm \, monic \, irreducible\, in\, } A}(1-\frac{P(x)^N}{P})^{-1}\in \mathbb C_\infty ^\times.$$
Write $N=\sum_{l=0}^k n_lq^l, n_l\in \{0, \ldots, q-1\}, n_k\not =0.$ For $l=0, \ldots, k,$ we have: 
$$\omega(t^{q^l})\mid_{t=x}\in \mathbb C_\infty^\times.$$
Therefore:
$$B_N(t, \theta)\mid_{t=x}\not =0.$$
This implies that, if $x \in \mathbb C_\infty$ is  a root of $B_N(t, \theta)$ then $v_\infty (x) \leq \frac{-1}{N}<0.$   Write in $\mathbb C_\infty[t]:$
$$B_N(t, \theta)= \lambda \prod_{j=1}^m(t-x_j), \lambda\in \mathbb F_p[\theta]\setminus\{0\} , \deg_\theta \lambda \leq r-1, x_1, \ldots, x_m \in \mathbb  C_\infty, $$
where $m=\deg_t \beta_N(t, \theta).$
Then:
$$\theta^r= (-1)^m \lambda \prod_{j=1}^m x_j.$$
Therefore:
$$\deg_\theta \lambda -r\leq \frac{-m}{N}.$$
We finally get:
$$\deg_tB_N(t, \theta) \leq (r-\deg_\theta \lambda)N\leq rN.$$
Since $B_N(t, \theta)$ is a monic polynomial in $\theta,$ the total degree in $t, \theta$ of $B_N(t, \theta)$ is less than or equal to  $\deg_tB_N(t, \theta)+r-2.$\par
\noindent Write:
$$B_N(t, \theta)= \alpha F, \alpha \in \mathbb F_p[\theta]\setminus\{0\},$$
where $F$ is a primitive polynomial (as a polynomial in $t$). In particular $\alpha$ must divide $\theta^r$ and $ B_N(1, \theta)$ in $\mathbb F_p[\theta].$ By Lemma \ref{LemmaATR5}, this implies that $\alpha \in \mathbb F_p^\times.$ 
\end{proof}
\begin{remark}  Let $f(\theta)$ be a monic irreducible polynomial in $\mathbb F_p[\theta].$  Let $P_1, \ldots, P_m$ be the monic irreducible polynomials in $A$ such that $f(\theta)=P_1\cdots P_m .$  We can order them such that if $d=\deg_\theta f,$ then for $i=1, \ldots, m,$ we have:
$$P_i=\sigma^{i-1} (P_1),$$
where $\sigma: A \rightarrow A$ is the morphism of $\mathbb F_p[\theta]$-algebras such that $\forall x\in \mathbb F_q,$ $\sigma(x)=x^p,$ and $m=[\mathbb F_{p^d} \cap \mathbb F_q: \mathbb F_p].$ Let $N\geq 1,$ and let's set:
$$\chi_N(f)=\prod_{i=1}^m(P_i(\theta)-P_i(t)^N)- f(\theta)\in \mathbb F_p[t, \theta].$$
We have:
$$\deg_t\chi_N(f)= N\deg_\theta f,$$
$$\deg_\theta \chi_N(f)\leq \deg_\theta f-\deg_\theta P_1\leq \deg_\theta (f) (1- \frac{\log(p)}{\log(q)}).$$

If $f_1, \ldots, f_n$ are $n$ irreducible monic polynomials in $\mathbb F_p[\theta],$ we set:
$$\chi_N(f_1\cdots f_m)=\prod_{l=1}^n \chi_N(f_l)\in \mathbb F_p[t, \theta].$$
Thus in $\mathbb F_p[t]((\frac{1}{\theta}))$:
$$L_N(t)=\sum_{d\geq 0} \sum_{a\in \mathbb F_p[\theta]_{+,d}}\frac{\chi_N(a)}{a}.$$
Observe that $\sum_{d\geq 0} \sum_{a\in \mathbb F_p[\theta]_{+,d}}\frac{\chi_N(a)}{a}$ converges on $\{x\in \mathbb C_\infty, v_\infty(x)>\frac{-1}{N}\}$ and does not vanish. Therefore on $\{x\in \mathbb C_\infty, v_\infty(x)>\frac{-1}{N}\}$ :
$$\sum_{d\geq 0} \sum_{a\in \mathbb F_p[\theta]_{+,d}}\frac{\chi_N(a)}{a}=\frac{1}{\theta^{\frac{\ell_q(N)-q}{q-1}}}B_N(t, \theta) \prod_{l=0}^k \prod_{j\geq 0} (1-\frac{t^{q^l}}{\theta^{q^j}})^{n_l},$$
where $N\equiv 1\pmod{q-1}, \ell_q(N)\geq q, N=\sum_{l=0}^k n_l q^l, n_l\in \{0, \ldots, q-1\}, n_k\not =0.$
\end{remark}

Let $s\geq 2,$ $s\equiv 1\pmod{q-1}.$ Recall that we have set:
$$\mathbb B_s=(-1) ^{\frac{s-1}{q-1}} \frac{ L_s \omega(t_1)\ldots \omega(t_s)}{\widetilde{\pi}}\in \mathbb T_s,$$
where $\mathbb T_s$ is the Tate algebra in the indeterminates $t_1, \ldots, t_s$ with coefficients in $\mathbb C_\infty$, and for $i=1, \ldots, s,$ $\omega(t_i)=\lambda_\theta \prod_{j\geq 0} (1-\frac{t_i}{\theta^{q^j}})^{-1}$. 
For $m\in \mathbb N,$ we denote by $BC_m\in K$ the $m$th Bernoulli-Carlitz number (\cite{GOS}, chapter 9, paragraph 9.2). 
\begin{proposition}\label{PropositionATR1} ${}$\par
\noindent 1)Let $N\geq 1, N\equiv 1\pmod{q-1},$ $\ell_q(N)\geq q.$ Recall that $r=\frac{\ell_q(N)-q}{q-1}.$ Let $d\geq 1$ such that $q^d>N, $ then we have the following equality in $\mathbb C_\infty:$
$$\frac{B_N(\theta, \theta^{q^d})}{\prod_{l=0}^k \prod_{n=0, n\not =l}^{d-1} (\theta^{q^l}-\theta^{q^n})^{n_l}}=(-1)^{r} \frac{BC_{q^d-N}}{\Pi(N)\Pi(q^d-N)},$$
where $\Pi(.)$ is the Carlitz factorial, and $N=\sum_{l=0}^k n_l q^l, n_l\in \{0, \ldots, q-1\}$.\par
\noindent 2) Let $N\geq 2, N\equiv 1\pmod{q-1}.$ Let $P$ be a monic irreducible polynomial in $A$ of degree $d\geq 1$ such that $q^d>N.$ Then $BC_{q^d-N} \equiv 0\pmod{P}$ if and only if $B_N(\theta,\theta)\equiv 0\pmod{P}.$
\end{proposition}
\begin{proof} ${}$\par
\noindent 1) The first assertion of the Proposition is a consequence  of the proof of Theorem 2 in \cite{AP}. For the convenience of the reader, we give  the proof of this result. Let $s\geq q,$  $ s\equiv 1\pmod{q-1}.$ 
Then (\cite{APT}, Lemma 7.6), we have that  $\mathbb B_s\in \mathbb F_q[t_1, \ldots, t_s, \theta]$ is a monic polynomial in $\theta$ of degree $\frac{s-q}{q-1}.$ \par
Let $\tau:\mathbb T_s\rightarrow \mathbb T_s$  be the continuous morphism of $\mathbb F_q[t_1, \ldots, t_s]$-algebras such that $\forall x\in \mathbb C_\infty, \tau(x)=x^q.$ Since $\lambda_\theta^q= -\theta \lambda_\theta, $ we have:
$$\tau(\omega(t_i))=(t_i-\theta) \omega(t_i), i=1, \ldots,s.$$
Let $d\geq 1,$ we get:
$$(-1)^{\frac{s-1}{q-1}}\frac{\tau^d(L_s)}{\widetilde{\pi}^{q^d}} \omega(t_1)\ldots \omega(t_s)= \frac{\tau^d(\mathbb B_s)}{\prod_{l=0}^{d-1}(t_1-\theta^{q^l})\cdots (t_s-\theta^{q^l})}.$$
Recall that, by formula (24) in \cite{PEL}, we have:
$$ (t_i-\theta^{q^l})\omega(t_i)\mid_{t=\theta^{q^l}}=-\frac{\widetilde{\pi}^{q^l}}{D_l} .$$
Let $l_1, \ldots, l_s\in \mathbb N,$ we get:
$$(-1)^{\frac{s-1}{q-1}}\frac{\tau^d(L_s)}{\widetilde{\pi}^{q^d}} (t_1-\theta^{q^{l_1}})\omega(t_1)\ldots (t_s-\theta^{q^{l_s}})\omega(t_s)= \frac{\tau^d(\mathbb B_s) (t_1-\theta^{q^{l_1}})\cdots (t_s-\theta^{q^{l_s}})}{\prod_{l=0}^{d-1}(t_1-\theta^{q^l})\cdots (t_s-\theta^{q^l})}.$$
Now, let $N\geq1$ such that $\ell_q(N)=s$ (observe that $N\equiv 1\pmod{q-1}$ and  $\ell_q(N)\geq q$). Write $N=\sum_{l=0}^k n_l q^l, n_0, \ldots n_k\in \{0, \ldots, q-1\},$ $n_k\not =0.$ Let $d\geq k+1.$  We get:\par
\noindent $(-1)^{\frac{s-1}{q-1}}\frac{\sum_{u\geq 0}\sum_{a\in A_{+,u}}\frac{a(t)^N}{a^{q^d}}}{\widetilde{\pi}^{q^d}} (\prod_{l=0}^k((t^{q^l}-\theta^{q^l})\omega(t^{q^l}))^{n_l})= \frac{B_N(t, \theta^{q^d})}{\prod_{l=0}^k\prod_{n=0, n\not = l}^{d-1}(t^{q^l}-\theta^{q^n})^{n_l}}.$\par
We get:
$$\frac{B_N(t, \theta^{q^d})}{\prod_{l=0}^k\prod_{n=0, n\not = l}^{d-1}(t^{q^l}-\theta^{q^n})^{n_l}}\mid_{t=\theta}= (-1)^{\frac{\ell_q(N)-q}{q-1}}\frac{BC_{q^d-N}}{\Pi(N) \Pi(q^d-N)}.$$

\noindent 2) The result is well-known for $\ell_q(N)=1$ (this is a consequence of the definition of the Bernoulli-Carlitz numbers and \cite{GOS}, Lemma 8.22.4). Thus, we will assume $\ell_q(N)\geq q.$ The assertion is then a consequence of the fact that:
$$B_N(\theta, \theta)\equiv B_N(\theta, \theta^{q^d})\pmod{P}.$$
\end{proof}
We have already mentioned that  $\mathbb B_s\in \mathbb F_q[t_1, \ldots, t_s, \theta]$ is a monic polynomial in $\theta$ of degree $\frac{s-q}{q-1}$ (\cite{APT}, Lemma 7.6). Let's observe that we have:
\begin{lemma}\label{LemmaREC} For $s\geq 2q-1,$ $s\equiv 1\pmod{q-1},$  we have:
$$\mathbb B_s(t_1, \ldots, t_{s-(q-1)}, 0, \ldots, 0)= (\theta -t_1\cdots t_{s-(q-1)})\mathbb B_{s-(q-1)}(t_1, \ldots, t_{s-(q-1)}).$$
More generally, if $\zeta$ is in the algebraic closure of  $\mathbb F_q$ in $\mathbb C_\infty,$ let $P$ be the monic irreducible polynomial in $A$ such that $P(\zeta)=0.$ Let $s\equiv 1\pmod{q-1},$ $s\geq q+q^d-1,$ where $d$ is the degree of $P.$ Write $s'=s-(q^d-1)$. We have:
$$\mathbb B_s(t_1, \ldots, t_{s')}, \zeta, \ldots, \zeta)= (P -P(t_1)\cdots P(t_{s'}))\mathbb B_{s'}(t_1, \ldots, t_{s'}).$$
\end{lemma}
\begin{proof}  The polynomial $ \mathbb B_s(t_1, \ldots, t_{s-(q-1)}, 0, \ldots, 0)$ is equal to:
$$(-\theta)(-1) ^{\frac{s-1}{q-1}} \frac{  \omega(t_1)\ldots \omega(t_{s-(q-1)})}{\widetilde{\pi}}\,\prod\limits_{\substack{P{\rm \, monic \, irreducible \, in\,  }A, \\ P\not =\theta}}(1-\frac{P(t_1)\cdots P(t_{s-(q-1)})}{P})^{-1}  .$$
The proof of the second  assertion of the Lemma is similar, using  \cite{AP2}, Theorem 2.9, and the properties of Gauss-Thakur sums (\cite{THA}).
\end{proof}

\begin{lemma}\label{Lemma1}
$$\mathbb B_{2q-1}=\theta -\sum_{1\leq i_1< \ldots <i_q\leq 2q-1}t_{i_1} \cdots t_{i_q}.$$
\end{lemma}
\begin{proof} Let $\mathbb T_{2q-1}(K_\infty)$ be the Tate algebra in the variable $t_1, \ldots, t_{2q-1}$ with coefficients in $K_\infty.$ Then:
$$\mathbb T_{2q-1}=\frac{1}{\theta}A[t_1, \ldots, t_{2q-1}]\oplus N,$$
where $N= \{ f\in \mathbb T_{2q-1}, v_\infty(f)\geq 2\}.$ Let $\phi: A\rightarrow A[t_1, \ldots, t_{2q-1}]\{Ê\tau\}$
 be the morphism of $\mathbb F_q$-algebras give by $\phi_\theta= (t_1-\theta)\cdots (t_{2q-1}-\theta)\tau +\theta.$ Then by the results in \cite{APT}, we have:
 $$N=\exp_\phi( \mathbb T_{2q-1}),$$
 and $H_\phi: =\frac{\mathbb T_{2q-1}}{A[t_1, \ldots, t_{2q-1}]\oplus N}$ is a free $\mathbb F_q[t_1, \ldots, t_{2q-1}]$-module of rank one generated by $\frac{1}{\theta}.$ Furthermore:
 $$\mathbb B_s=\det_{\mathbb F_q[t_1, \ldots, t_{2q-1}][Z]}(Z{\rm Id}-\phi_\theta\mid_{H_\phi\otimes_{\mathbb F_q[t_1, \ldots, t_{2q-1}]}\mathbb F_q[t_1, \ldots, t_{2q-1}][Z]})\mid_{Z=\theta}.$$
 Now, observe that:
 $$\phi_\theta (\frac{1}{\theta}) \equiv \frac{\sum_{1\leq i_1< \ldots <i_q\leq 2q-1}t_{i_1} \cdots t_{i_q}}{\theta}\pmod{A[t_1, \ldots, t_{2q-1}]\oplus N}.$$
 The Lemma follows.
\end{proof}
 For $\underline{m}=(m_1, \ldots, m_d)\in \mathbb N^d,$ we set $m_0: =s-(m_1+\cdots +m_d),$ and:
$$\sigma_s(\underline{m})=\sum\prod_{u=1}^d\prod_{i\in J_u} t_i^u,$$
where the sum runs through the disjoint unions $J_1\bigsqcup \cdots \bigsqcup J_k\subset \{ 1, \ldots, s\}$ such that $\mid J_u\mid =m_u,$ $u=1, \ldots, d.$ Notice in particular that $\sigma_s(\underline m)= 0$ if $m_1+\cdots +m_d>s$, that is, if $m_0<0$. To give an example, the above lemma shows that $\mathbb B_{2q-1}=\theta - \sigma((q))$.

\begin{lemma}\label{Lemma2} Let $\underline{m}\in \mathbb N^d.$ We have :
$$\sigma_s (\underline{m})\mid_{t_s=0, \ldots, t_{s-(q-2)}=0}= \sigma_{s-(q-1)} (\underline{m}).$$
In particular, if $m_0<q-1,$ we have:
$$\sigma_s (\underline{m})\mid_{t_s=0, \ldots, t_{s-(q-2)}=0}=0.$$
\end{lemma}
\begin{proof} This is a straight computation.
\end{proof}
Let $\rho: \mathbb F_p[t_1, \ldots, t_s]\rightarrow \mathbb N\cup \{ +\infty\}$ be the function given by:\par
\noindent - if $f=0,$ $\rho(f)=+\infty,$\par
\noindent - if $f\not =0,$ $f=\sum \alpha_{j_1, \ldots, j_s}t_1^{j_1}\cdots t_s^{j_s},$ $\alpha_{j_1, \ldots, j_s}\in \mathbb F_p,$ then $\rho (f)={\rm Inf}\{ j_1+\ldots +j_s,\, \alpha_{j_1, \ldots, j_s}\not =0\}.$\par
 Let's write:
$$ \mathbb B_s= \sum_{ i=0}^rB_{i,s} \theta^{r-i},$$
where $B_{i,s}\in \mathbb F_p[t_1, \ldots, t_s]$ is a symmetric polynomial, and $B_{0,s}=1.$

\begin{proposition}\label{Proposition1}
For $i=1, \ldots, r,$ we have:
$$\rho (B_{i,s})\geq  i(q-1)+1.$$
\end{proposition}
\begin{proof} By Lemma \ref{Lemma1}, this is true for $r=1,$ thus we can assume that $r\geq 2.$ The proof is by induction on $r.$  Recall that by Lemma \ref{LemmaATR5}, we have:
$$\mathbb B_s\mid_{t_1=\ldots= t_s=0}= \theta^r.$$
Thus, for $i=1, \ldots, r,$ we can write:
$$B_{i,s}=\sum_{\underline{m}\in S} x_{i, \underline{m}} \sigma_s (\underline{m}), x_{i, \underline{m}}\in \mathbb F_p,$$
where $S=\{(m_1, \dots, m_s)\in \mathbb N^s, 1\le m_1+\cdots+m_d\le s\}$.
Set:
$$\widetilde{B_{i,s}}=B_{i,s}-\sum_{\underline{m}\in S, m_0<q-1} x_{i, \underline{m}} \sigma_s (\underline{m}).$$
Then:
$$\rho(B_{i,s}-\widetilde{B_{i,s}})\geq r(q-1)+2.$$
Therefore we have to prove:
$$\rho(\widetilde{B_{i,s}})\geq i(q-1)+1.$$
Observe that, by Lemma \ref{Lemma2}, we have:
$$B_{i,s}\mid_{t_s= \ldots= t_{s-(q-2)}=0}= \widetilde{B_{i,s}}\mid_{t_s= \ldots= t_{s-(q-2)}=0}= \sum_{\underline{m}\in S, m_0\geq q-1} x_{i, \underline{m}} \sigma_{s-(q-1)} (\underline{m}).$$
By Lemma \ref{LemmaREC}, we have:
$$\mathbb B_s\mid_{t_s= \ldots= t_{s-(q-2)}=0}= (\theta-t_1\cdots t_{s-(q-1)}) \mathbb B_{s-(q-1)}.$$
We therefore get, for $i=1, \ldots, r:$
$$\widetilde{B_{i,s}}\mid_{t_s= \ldots= t_{s-(q-2)}=0}= B_{i, s-(q-1)}- t_1\cdots t_{s-(q-1)}B_{i-1, s-(q-1)},$$
where we have set $ B_{r, s-(q-1)}=0.$ Now, by the induction hypothesis:
$$\rho(B_{i, s-(q-1)}- t_1\cdots t_{s-(q-1)}B_{i-1, s-(q-1)})\geq i(q-1)+1.$$
Thus:
$$\rho(B_{i,s})\geq i(q-1)+1.$$
\end{proof}

\begin{corollary}\label{Corollary1} Let $N\equiv 1\pmod{q-1},$ $\ell_q(N)\geq 2q-1.$ Then $\forall a\in \overline{\mathbb FÑ}_q[\theta],$ we have:
$$B_N(t, \theta)\mid_{t=a}\not =0.$$
\end{corollary}
\begin{proof} By Proposition \ref{Proposition1}, we have:
$$B_N(t, \theta) -\theta^r \in t(t, \theta)^{r}.$$
The Corollary follows easily.
\end{proof}
\noindent {\sl Proof of Theorem \ref{TheoremATR-BC}}${}$\par
We can assume that $\ell_q(N) \geq q.$ By Lemma \ref{LemmaATR6}, the total degree in $t, \theta$ of $B_N(t, \theta)$ is strictly less than $(r+1)N,$  where $r=\frac{\ell_q(N)-q}{q-1}.$ Now, by Corollary \ref{Corollary1}:
$$B_N(\theta, \theta)\not =0.$$
Furthermore, $\deg_\theta B_N(\theta, \theta) <(r+1)N.$ Thus if $P$ is a monic irreducible polynomial in $A$ such that $\deg_\theta P \geq (r+1)d,$ we have:
$$B_N(\theta, \theta )\not \equiv 0\pmod{P}.$$
We conclude the proof of the Theorem by Proposition \ref{PropositionATR1}.

%%%%%%%%%%%%%%%%%%%%%%%%%%%%%%%%%%%%%%%%%%%%%%%%%%%%%%%%%%%%%%%%%%%%%%%%%
\section{Exceptional zeros and eigenvalues of certain $K$-endomorphisms}
%%%%%%%%%%%
\subsection{The $L$-series $L_N(t)$}${}$\par
Let $N\geq 1$ be an integer. Recall that 
$$L_N(t)=\sum_{d\geq 0}\sum_{a\in A_{+,d}}\frac{a(t)^N}{a} =\sum_{i\geq 0} \alpha_{i,N}(t) \theta^{-i}, \, \alpha_{i,N}(t) \in \mathbb T_t^\times.$$

Let $d\geq 1$ be an integer, we set for ${\bf k}= (k_0, \ldots, k_{d-1})\in \mathbb N^{d}$:\par
\noindent - $\ell({\bf k})=d,$\par
\noindent - $\mid {\bf k}\mid= k_0+\cdots +k_{d-1},$\par
\noindent - if $a=a_0+a_1\theta+\cdots +a_{d-1} \theta^{d-1} +\theta^d, a_i\in \mathbb F_q, i=0, \ldots, d-1,$ $a^{{\bf k}}=\prod_{i=0}^{d-1} a_i^{k_{i}}.$\par
${}$\par
Let's begin by a simple observation.
Let $d\geq 1$ and let ${\bf k}=(k_0, \ldots, k_{d-1})\in \mathbb N^{d}.$   Let $N\geq 1$ be an integer, we get:
$$\sum_{a\in A_{+,d}} a(t)^N a^{{\bf k}}= \sum_{a\in A_{+,d}}\sum_{{\bf m}\in \mathbb N^{d+1}, \mid{\bf m}\mid = N} C(N, {\bf m}) a^{{\bf k}}a^{{\bf m}}t^{   
m_1+2m_2+\cdots + dm_{d}} ,$$
where:
$$C(N, {\bf m})=\frac{N!}{m_0!\cdots m_{d}!}\in \mathbb F_p.$$
Recall that by Luca's Theorem $C(N, {\bf m})\not =0$ if and only if there is no carryover $p$-digits in the sum $N=m_0+\cdots+m_{d}.$ Furthermore, recall that, for $n\in \mathbb N,$ $\sum_{\lambda \in \mathbb F_q} \lambda^n \not = 0$ if and only if $n\equiv 0\pmod{q-1}$ and $n\geq 1.$ Thus,  for ${\bf m}\in \mathbb N^{d+1},$ $\sum_{a\in A_{+,d}} a^{{\bf m}}=0$ unless $(m_0, \dots, m_{d-1})\in ((q-1)(\mathbb N\setminus \{0\}))^d$ and in this latter case $\sum_{a\in A_{+,d}} a^{{\bf m}}=(-1)^d.$\par
Thus for $d,N\geq 1,$ $ {\bf k}\in \{0, \ldots, q-1\}^d,$ we denote by $U_{d}(N, {\bf k})$ the set of elements ${\bf m}\in \mathbb N^{d+1}$ such that:\par
\noindent - there is no carryover $p$-digits in the sum $N=m_0+\cdots+m_d,$\par
\noindent - for $n=0, \ldots, d-1,$ $m_n-k_n \in (q-1) \mathbb N.$  \par 
\noindent For ${\bf m}\in U_d(N, {\bf k}),$ we set:
$$\deg {\bf m}= m_1+2m_2+\cdots + d m_{d}.$$
An element ${\bf m}\in U_{d}(N,{\bf k})$ is called optimal if $\deg {\bf m}= {\rm Max}\{ \deg {\bf n}, {\bf n}\in U_d(N, {\bf k})\}.$ If $U_{d}(N,{\bf k})\not =\emptyset,$ the greedy element of $ U_{d}(N,{\bf k})$ is the element ${\bf m}= (m_0, \ldots, m_d)\in U_{d}(N,{\bf k})$ such that $(m_{d}, \ldots, m_1)$ is largest lexicographically.\par
Let ${\bf k}\in \mathbb N^d, d\geq 1.$ For $n=0, \ldots, d-1,$ let $\bar k_n \in \{ 0, \ldots, q-1\}$ be the least integer such that $k_n +\bar k_n \in (q-1)(\mathbb N\setminus\{0\}).$ We set:
$$\bar {\bf k}=(\bar k_0, \ldots, \bar k_{d-1}).$$
We get:
$$\sum_{a\in A_{+,d}} a(t)^N a^{{\bf k}}=(-1)^d\sum_{{\bf m} \in U_d(N, \bar {\bf k})}C(N, {\bf m}) t^{\deg {\bf m}}.$$
${}$\par
Let $N\geq 1$ be an integer and let $\ell_q(N)$ be the sum of digits of $N$ in base $q.$ Then we can write in a unique way:
$$N=\sum_{n=1}^{\ell_q(N)} q^{e_n}, \, e_1\leq e_2\leq \cdots \leq e_{\ell_q(N)}.$$
We set:
$$r={\rm Max} \{ 0, [\frac{\ell_q(N)-q}{q-1}]\}\in \mathbb N.$$

\begin{lemma}\label{LemmaATR2} We have:
$$\forall i\geq 0,\,  \alpha_{i,N}(t)=\sum_{\ell({\bf k})+w({\bf k})=i} (-1)^{\ell ({\bf k})}C_{{\bf k}}\sum_{a\in A_{+,\ell({\bf k})}} a(t)^N a^{{\bf k}} \in \mathbb F_p[t] ,$$
where $C_{{\bf k}}= (-1)^{\mid {\bf k}\mid} \frac{\mid {\bf k}\mid !}{k_0!\cdots k_{d-1}!} \in \mathbb F_p,$ $w({\bf k})= k_{d-1}+\cdots +(d-1)k_1+dk_0, $ for  ${\bf k}=(k_0, \ldots, k_{d-1}).$
\end{lemma}
\begin{proof}  Let $a\in A_{+,d}.$ We have:
$$\frac{1}{a}= \frac{1}{\theta^d}\sum_{{\bf k}\in \mathbb N^d}C_{{\bf k}} a^{{\bf k}} \frac{1}{\theta^{w({\bf k})}},$$
where $C_{{\bf k}}= (-1)^{\mid {\bf k}\mid} \frac{\mid {\bf k}\mid !}{k_0!\cdots k_{d-1}!} \in \mathbb F_p, $ $w({\bf k})= k_{d-1}+\cdots +(d-1)k_1+dk_0. $ 
Thus:
$$\sum_{a\in A_{+,d}}\frac{a(t)^N}{a} =\frac{(-1)^d}{\theta^d}\sum_{{\bf k}\in \mathbb N^d}C_{{\bf k}}  \frac{1}{\theta^{w({\bf k})}}\sum_{a\in A_{+,d}} a(t)^N a^{{\bf k}}.$$
Therefore: 
$$ \alpha_{i,N}(t)=\sum_{\ell({\bf k})+w({\bf k})=i} (-1)^{\ell ({\bf k})}C_{{\bf k}}\sum_{a\in A_{+,\ell({\bf k})}} a(t)^N a^{{\bf k}} \in \mathbb F_p[t].$$
\end{proof}
\begin{lemma}\label{LemmaATR3} Let $j\in \mathbb Z.$ Then $L_N(t)\mid_{t=\theta^{q^j}}=0$ if and only if $N\equiv 1\pmod{q-1},$  and $q^jN>1.$
\end{lemma}
\begin{proof} This comes from the following facts:\par
\noindent - $\forall n\geq 1,$ $\sum_{d\geq 0} \sum_{a\in A_{+,d}}\frac{1}{a^n}\not=0,$\par
\noindent - for $n\geq 0,$ $\sum_{d\geq 0} \sum_{a\in A_{+,d}}{a^n}=0$ if and only if $n\geq 1,$ $n\equiv 0\pmod{q-1}.$
 \end{proof}
We will need the following Lemma in the sequel:
\begin{lemma}\label{LemmaATR4}
Let $F(t)=\sum_{n\geq 0} \beta_n(t) \frac{1}{\theta^n}\in \mathbb F_q[t][[\frac{1}{\theta}]],\, \beta_n(t) \in \mathbb F_q[t],$ and let $\rho\in \mathbb R$ such that $F(t)$ converges on $\{ x\in \mathbb C_\infty, v_\infty(x)\geq \rho\}.$ Let $M\geq 1$ and set $F_M(t)=\sum_{n=0}^M\beta_n(t) \frac{1}{\theta^n} \in K_\infty[t].$  Let $\varepsilon \in \mathbb R, \varepsilon \geq \rho.$ Suppose that $F_M(t)$ has exactly $k\geq 1$ zeros in $\mathbb  C_\infty$ with valuation $\varepsilon.$  Then either $F(t)$ has $k$ zeros with valuation $\varepsilon$ or $F(t)$
 has at least  $\deg_tF_M(t)+1$ zeros with valuation $>\varepsilon.$ 
\end{lemma}
\begin{proof}  Let's assume that the side of the Newton polygon of $F_M(t)$ corresponding to the $k$ zeros of valuation $\varepsilon$  is not a portion of a side of the Newton polygon of $F(t),$ then $F(t)$ has a side of slope $-\varepsilon'<-\varepsilon$ with end point of abscissa $k'>\deg_tF_M(t).$ Thus the Newton polygon of $F(t)$ delimited by the vertical axis  of abscissas $0$ and $k'$ has only sides of slope $\leq -\varepsilon'.$ Thus $F(t)$ has $k'$ zeros of valuation $\geq \varepsilon'.$
\end{proof}

%%%%%%%%%%%
\subsection{An example}${}$\par
For the convenience of the reader, we treat a basic example:  $N=1.$ We set $\ell_0=1$ and for $d\geq 1,$ $\ell_d=(\theta-\theta^q)\cdots (\theta-\theta^{q^d}).$
\begin{lemma}\label{LemmaEX1} Let $d\geq 0.$ Then:
$$\sum_{a\in A_{+,d}}\frac{1}{a}= \frac{1}{\ell_d}.$$
\end{lemma}
\begin{proof} This is a well-known  consequence of a result of Carlitz (\cite{GOS}, Theorem 3.1.5). Let's give a proof for the convenience of the reader. We can assume that $d\geq 1.$ Set:
$$e_d(X)=\prod_{a\in A, \deg_\theta a<d}(X-a)\in A[X].$$
Then (\cite{GOS}, Theorem 3.1.5):
$$e_d(X)= \sum_{i=0}^d \frac{D_d}{D_i\ell_{d-i}^{q ^i}}X^{q^i},$$
where $D_0=1, $ and for $i\geq 1,$ $D_i=(\theta^{q ^i}-\theta)D_{i-1}^q.$ Now observe that:
$$\frac{\frac{d}{dX}(e_d(X-\theta^d))}{e_d(X-\theta^d)}\mid_{X=0}= -\sum_{a\in A_{+,d}}\frac{1}{a}.$$
Since $e_d(\theta^d)=D_d$ (\cite{GOS}, Corollary 3.1.7), we get the desired result.
\end{proof}
\begin{lemma}\label{LemmaEX2} Let $d\geq 0.$ then:
$$\sum_{a\in A_{+,d}}\frac{a(t)}{a} = \frac{(t-\theta) \cdots (t-\theta^{q^{d-1}})}{\ell_d}.$$
\end{lemma}
\begin{proof} We can assume that $d\geq 1.$  Set:
$$S=\sum_{a\in A_{+,d}}\frac{a(t)}{a}.$$
Then for $i=0, \ldots, d-1,$ we have:
$$S\mid_{t=\theta^{q^i}}=0.$$
Furthermore, by Lemma \ref{LemmaEX1}, $S$ has degree $d$ in $t$ and the coefficient of $t^d$ is $\frac{1}{\ell_d}.$ The Lemma follows.
\end{proof}
\begin{lemma}\label{LemmaEX3}
The edge points of the Newton polygon of $L_1(t)$ are $(d, q\frac{q^d-q}{q-1}), d\in \mathbb N.$
\end{lemma} 
\begin{proof} Let's write:
$$L_1(t)=\sum_{d\geq 0} S_d(t),$$
where:
$$S_d(t) =\sum_{a\in A_{+,d}}\frac{a(t)}{a}.$$
Let $d\in \mathbb N$ and let $d'>d.$ Let $x\in K$ be the coefficient of $t^d$ in $S_{d'}(t).$ Then, by Lemma \ref{LemmaEX2}, we get:
$$v_\infty(x)\geq -v_\infty(\ell_{d'})-q^{d}-\cdots -q^{d'-1}> q\frac{{q^d}-1}{q-1}.$$
Thus, if we write:
$$L_1(t)=\sum_{d\geq 0} \alpha_d t^d \in K_\infty[[t]], \alpha_d\in K_\infty, d\in \mathbb N,$$
by the above observation and again by Lemma \ref{LemmaEX2}, we get:
$$\forall d\geq 0, v_\infty(\alpha_d)=q\frac{{q^d}-1}{q-1}.$$
\end{proof}
This latter Lemma implies the following formula due to F. Pellarin (\cite{PEL}, Theorem 1):
\begin{proposition}\label{PropositionEX1}
Let $\lambda_\theta\in \mathbb C_\infty^\times$ be a fixed $(q-1)$th root of $-\theta.$ Set:
$$\widetilde{\pi}= \lambda_\theta \theta \prod_{j\geq 1} (1-\theta^{1-q^j})^{-1}\in \mathbb C_\infty^\times.$$
Then:
$$(\theta-t) L_1(t) =\frac{\widetilde{\pi}}{\lambda_\theta} \prod_{j\geq 0} (1-\frac{t}{\theta^{q^j}}).$$
\end{proposition}
\begin{proof} We observe that:
$$\forall n\geq 1, L_1(t)\mid_{t=\theta^{q^n}}= 0,$$
$$L_1(\theta)=1.$$
By Lemma \ref{LemmaEX3}, the entire function $(t-\theta) L_1(t)$ has simple zeros in $K_\infty$ and if $x\in K_\infty$ is such a zero, $v_\infty (x) \in \{ -q^i, i\in \mathbb N\}.$ Thus, there exists $\alpha \in \mathbb C_\infty^\times$ such that:
$$(t-\theta) L_1(t) =\alpha \prod_{j\geq 0} (1-\frac{t}{\theta^{q^j}}).$$
But, observe that:
$$\frac{\widetilde{\pi}}{\lambda_\theta} \prod_{j\geq 1} (1-\frac{t}{\theta^{q^j}})\mid_{t=\theta}=\theta.$$
Therefore:
$$\alpha= \frac{-\widetilde{\pi}}{\lambda_\theta}.$$
\end{proof}

\subsection{Eigenvalues  and Bernoulli-Carlitz numbers}${}$\par

In this paragraph, we slightly change our point of view. Let $t$ be an indeterminate over $\mathbb C_\infty$ and let $\varphi: \mathbb C_\infty[[\frac{1}{t}]]\rightarrow\mathbb C_\infty[[\frac{1}{t}]]$ be the continuous (for the $\frac{1}{t}$-adic topology) morphism of $\mathbb C_\infty$-algebras such that $\varphi(t)=t^q.$ We first recall some consequences of the work of  F. Demeslay's  (see the appendix of \cite{APT} or \cite{DEM}) generalizing the work of L. Taelman (\cite{TAE1.1}).\par
Let $N\geq 1, N\equiv 1\pmod{q-1},$ $\ell_q(N)\geq q.$ Write $N=\sum_{l=0}^kn_l q^l, $ $n_l\in \{0, \ldots, q-1\}, l=0, \ldots, k, $ and $n_k\not =0.$ We set $B=K[t].$  Let $\phi^{(N)}: B\rightarrow B\{\varphi\}$ be the morphism of $K$-algebras given by:
$$\phi^{(N)}_t= (\prod_{l=0}^k (\theta^{q^l }-t)^{n_l}) \varphi +t.$$
Since $t$ is transcendental over $\mathbb F_q,$ there exists a unique "power series" $\exp_{\phi^{(N)}}\in K(t)\{\{\varphi\}\}$ such that:
$$\exp_{\phi^{(N)}}\equiv 1\pmod{\varphi},$$
$$\exp_{\phi^{(N)}} t=\phi^{(N)}_t \exp_{\phi^{(N)}}.$$
One can easily see that:
$$\exp_{\phi^{(N)}}=\sum_{j\geq 0} \frac{\prod_{l=0}^k (\prod_{n=0}^{j-1}(\theta^{q^l}- t^{q^n}))^{n_l}}{\prod_{n=0}^{j-1}(t^{q^n}-t^{q^j})}\varphi^j.$$
In particular $\exp_{\phi^{(N)}}$ induces a continuous $K$-linear endomorphism of  $K((\frac{1}{t}))$ which is an isometry on a sufficiently small neighborhood of zero (for the $\frac{1}{t}$-adic topology). Let's set:
$$H(\phi^{(N)})= \frac{K((\frac{1}{t}))}{(B+ \exp_{\phi^{(N)}}(K((\frac{1}{t})))}.$$
Then $H(\phi^{(N)})$ is a finite $K$-vector space and a $B$-module via $\phi.$ Let's denote by $[H(\phi^{(N)})]_B$ the monic generator (as a polynomial in $t$) of the Fitting ideal of the $B$-module $H(\phi^{(N)}),$ i.e.:
$$[H(\phi^{(N)})]_B=\det_{K[Z]}(Z{\rm Id}-\phi^{(N)}_t\mid_{H(\phi^{(N)})})\mid_{Z=t}.$$
As in \cite{APT}, Proposition 7.2,  one can prove that:
$$\dim_K H(\phi^{(N)})= \frac{\ell_q(N)-q}{q-1},$$
$$\{ x\in K((\frac{1}{t})),\exp_{\phi^{(N)}}(x)\in B\}= \frac{\bar \pi}{t^{\frac{\ell_q(N)-q}{q-1}}\prod_{l=0}^k\bar { \omega} (\theta^{q^l})^{n_l}}B,$$
where:
$$\bar \pi = \prod_{j\geq 1} (1-t^{1-q^j})^{-1} \in \mathbb F_p[[\frac{1}{t}]]^\times,$$
$$\bar\omega(\theta^{q^l})= \prod_{j\geq 0} (1-\frac{\theta^{q^l}}{t^{q^j}})^{-1}\in A[[\frac{1}{t}]]^\times.$$ 
Furthermore, if we set:
$$\mathcal L_N(t)=\prod_{P(t) {\rm \, monic\,  irreducible\,  polynomial\,  of\, }\mathbb F_q[t]}\frac{[\frac{B}{P(t)B}]_B}{[\phi^{(N)}(\frac{B}{P(t)B})]_B},$$
then, by the appendix of \cite{APT}, $\mathcal L_N(t)$ converges in $K((\frac{1}{t})),$ and:
$$[H(\phi^{(N)})]_B \frac{\bar \pi}{t^{\frac{\ell_q(N)-q}{q-1}}\prod_{l=0}^k\bar { \omega} (\theta^{q^l})^{n_l}} =\mathcal L_N(t).$$
Now, one can compute $\mathcal L_N(t)$ as in \cite{APT}, paragraph 5.3, and we get:
$$\mathcal L_N(t)= \sum_{d\geq 0} \sum_{a\in A_{+,d}}\frac{a^N}{a(t)}\in A[[\frac{1}{t}]]^\times.$$
Therefore:
$$[H(\phi^{(N)})]_B=B_N(\theta,t).$$
We warn the reader not to confuse $B_N(\theta,t)$ and $B_N(t,\theta)$, here and in the sequel of the paper, since we will be interested in those two polynomials. Recall that $r=\frac{\ell_q(N)-q}{q-1}.$ Let $\alpha_1(N), \ldots, \alpha_r(N)\in \mathbb C_\infty$ be the eigenvalues (counted with multiplicity) of the $K$-endomorphism of $H(\phi^{(N)})$: $\phi^{(N)}_t.$ We get:
$$B_N(\theta, \theta)= \prod_{j=1}^r (\theta-\alpha_j(N)).$$

Recall that $\mathcal L_N(t)= \sum_{d\geq 0} \sum_{a\in A_{+,d}}\frac{a^N}{a(t)}\in A[[\frac{1}{t}]].$ By Lemma \ref{LemmaATR1}, $\mathcal L_N(t)$ converges on $\{ x\in \mathbb C_\infty, v_\infty(x)<0\}.$ Let's write $N=\sum_{l=0}^k n_lq^l , n_l\in \{0, \ldots, q-1\}, n_k\not =0.$ Set $S=\{ \theta^{q^i}, i\leq k\}.$ Then, by Lemma \ref{LemmaATR3},  the elements of $S$ are zeros of $\mathcal L_N(t)$ The elements of $S$ are called the trivial zeros of $\mathcal L_N(t).$ A zero of $\mathcal L_N(t)$ which does  belong to $\{ x\in \mathbb C_\infty, v_\infty(x)<0\}\setminus S$ will be called an exceptional zero of $\mathcal L_N(t).$ It is clear that the exceptional zeros of $\mathcal L_N(t)$ are roots of $B_N(\theta, t)$ with the same multiplicity. Our aim in the remaining of the article is to study the following problem:
\begin{problem}\label{Conjecture} ${}$\par
Let $N\geq 2, $ $N\equiv 1\pmod{q-1}, \ell_q(N)\geq q.$  Then all the eigenvalues of  $\phi^{(N)}_t$ (viewed as a $K$-endomorphism of $H(\phi^{(N)})$) are simple and belong to $\mathbb F_p((\frac{1}{\theta}))$.
\end{problem} 
Theorem \ref{TheoremATR-BC} implies that $\theta$ is not an eigenvalue of $\phi^{(N)}_t$. We presently do not know whether another trivial zero of $\mathcal L_N(t)$ can be an eigenvalue of $\phi^{(N)}_t$.On the other side, the above problem implies that the exceptional zeros of $\mathcal L_N(t)$ are simple. Observe that, by Lemma \ref{LemmaATR5}, the above Problem has an affirmative answer for $q\leq \ell_q(N)\leq 2q-1.$ \par

%%%%%%%%%%%%%%%%%%%%%%%%%%%%%%%%%%%%%%%%%%%%%%%%%%%%%%%%%%%%%%%%%%
\section {Answer to  Problem \ref{Conjecture}  for $q=p$}
In this section we  give an affirmative answer to  Problem \ref{Conjecture} in the case $q=p.$ By Proposition \ref{PropositionATR1} and Proposition \ref{PropositionATR6} below, this implies Theorem \ref{TheoremATR-BC}. For the convenience of the reader, we have tried to keep the text of this section as self-contained as possible.\par
In this  section $q=p.$\par
\subsection{Preliminaries}$ $\par
Lemma \ref{LemmaATR7} and Proposition \ref{PropositionATR2} below are  slight generalizations of the arguments used in the  proof of Theorem 1 in \cite{DIV}.\par

\begin{lemma}\label{LemmaATR7} Let $d,N\geq 1$ and ${\bf k}=(k_0, \ldots, k_{d-1})\in \mathbb \{ 0, \ldots, p-1\}^d.$ We assume that $\mid {\bf k}\mid  \leq \ell_p(N).$ For $1\leq i\leq d,$ we set : $\sigma_i=\sum_{n=0}^{i-1}  k_n.$ We also set $\sigma_0=0$ and $\sigma_{d+1}=\ell_p(N)$.  Let ${\bf m}=(m_0, \ldots, m_{d})\in \mathbb N^{d+1}$ be the element defined as follows:
$$n=0, \ldots, d, \ m_n=\sum_{l=\sigma_n+1}^{\sigma_{n+1}} p^{e_l}.$$
Then:
$${\bf m}\in U_d(N, {\bf k}).$$
Furthermore ${\bf m}$ is the greedy element of $U_d(N, {\bf k}).$ In particular $U_d(N, {\bf k})\not = \emptyset$ if and only if $\mid {\bf k}\mid  \leq \ell_p(N).$ 
\end{lemma}
\begin{proof} Observe that $\sigma_d = \mid {\bf k}\mid  \leq \ell_p(N).$ Thus ${\bf m}$ is well-defined. It is then straightforward to verify that ${\bf m}\in U_d(N, {\bf k})$ and that ${\bf m}$ is the greedy element of $U_d(N, {\bf k}).$\par
\noindent Now assume that  $U_d(N, {\bf k})\not = \emptyset .$ Let ${\bf m}'\in U_d(N, {\bf k}).$ Then:
$$ n= 0, \ldots, d-1, \ell_p(m_n')\equiv  k_n\pmod{p-1}.$$
This implies:
$$ n= 0, \ldots, d-1, \ell_p(m_n')\geq   k_n.$$
Thus:
$$\ell_p(N)\geq \sum_{n=0}^{d-1} \ell_p(m_n')\geq \mid {\bf k}\mid.$$
\end{proof}

\begin{proposition}\label{PropositionATR2}
Let $d,N\geq 1$ and ${\bf k}\in \mathbb N^d.$ We assume that $\mid \bar {\bf k}\mid  \leq \ell_p(N).$ Then $U_d(N, \bar {\bf k})$ contains a unique optimal element which is equal to the greedy element of $U_d(N, \bar {\bf k}).$  In particular $\sum_{a\in A_{+,d}} a(t)^N a^{{\bf k}}\not =0$ if and only if $\mid \bar {\bf k}\mid  \leq \ell_p(N).$
\end{proposition}
\begin{proof} Let ${\bf u}=(u_0, \ldots, u_{d})$ be the greedy element of $U_d(N, \bar{\bf k}).$ Let ${\bf m}\in U_d(N,\bar {\bf k})$ such that ${\bf m} \not = {\bf u}.$  We will show that ${\bf m}$ is not optimal.\par
Write $c_n=\ell_p(m_n),$ $n=0, \ldots d-1.$  Then :
$$n=0, \ldots, d-1, c_n \geq \bar k_n, c_n\equiv \bar k_n \pmod{p-1}.$$
For $n=0, \ldots, d-1,$ there exist $f_{n,1}\leq\cdots  \leq f_{n, c_n}$ such that we can write in a unique way:
$$m_n=\sum_{l=1}^{c_n} p^{f_{n,l}}.$$
{\sl Case 1) There exists an integer $j$, $0\leq j \leq d-1$, such that $c_{j}>\bar k_j.$}\par
\noindent Let ${\bf m}'\in \mathbb N^{d+1}$ be defined as follows:\par
\noindent - $m'_n=m_n $ for $0\leq n \leq d-1$, $n\not =j,$\par
\noindent -$m'_j= \sum_{l=1}^{\bar k_j} p^{f_{j,l}}.$\par
\noindent -$m'_d=N-m_0'-\cdots- m_{d-1}' = m_d+m_j-m_j' .$\par
\noindent Then ${\bf m}'\in U_d(N, {\bf k})$ and:
$$\deg  {\bf m}'= \deg {\bf m} + (d-j)(m_j-m_j')>\deg {\bf m} .$$
Thus ${\bf m}$ is not optimal.\par
\noindent {\sl Case 2) For $n=0, \ldots, d-1,$ $c_n=\bar k_n.$}\par
\noindent Let $j\in \{0, \ldots d-1\}$ be the smallest integer such that $m_j \not = u_j.$ Then, by the construction of ${\bf u},$ we have:
$$m_j > u_j. $$
Thus there exists an integer $l$ such that  the number of times $p^l$  appears in the sum of $m_j$ as $\bar k_j$   powers of $p$  is strictly greater than the number of times it appears in the sum of $u_j$ as $\bar k_j$ powers of $p.$ Also, there exists an integer $v$ such that  the number of times $p^v$  appears in the sum of $u_j$ as $\bar k_j$   powers of $p$  is strictly greater than the number of times it appears in the sum of $m_j$ as $\bar k_j$ powers of $p.$ Thus there exists an integer $t>j$ such that $p^v$ appears in the sum of $m_t$ as $\ell_p(m_t)$ powers of $p$. We observe that, by the construction of ${\bf u},$ we can choose $v$ and $l$ such that $ v< l.$  Now set:\par
\noindent - for $n=0, \ldots, d,$ $n\not = j, n\not = t,$ $m'_n=m_n,$\par
\noindent - $m'_j= m_j-p^l +p^v,$\par
\noindent - $m'_t= m_t-p^v+p^l.$\par
\noindent Let ${\bf m}'= (m'_0, \ldots, m'_{d-1})\in \mathbb N^d.$ Then ${\bf m}'\in U_d(N, \bar{\bf k})$ and:
$$\deg {\bf m}'=\sum_{l=0}^{d} lm'_l=\deg {\bf m} + (t-j)(p^l-p^v) >\deg {\bf m} .$$
Thus ${\bf m}$ is not optimal.
\end{proof}
We have the following key result:
\begin{proposition}\label{PropositionATR3}
Let $d,N\geq 1$ and ${\bf k}\in \mathbb N^d.$ We assume that $\ell_p(N)\geq p$ and that  $d(p-1)  \leq \ell_p(N)-p.$
Then:
$$N(d-1)< \deg_t \sum_{a\in A_{+,d}} a(t)^N a^{{\bf k}} \leq Nd.$$
\end{proposition}
\begin{proof}   It is clear that $\deg_t \sum_{a\in A_{+,d}} a(t)^N a^{{\bf k}} \leq Nd.$ Observe that $\mid \bar {\bf k}\mid \leq d(p-1).$ Let ${\bf m}$ be the greedy element of $U_d(N,\bar{\bf k}).$  By Proposition \ref{PropositionATR2},  we have:
$$\deg_t \sum_{a\in A_{+,d}} a(t)^N a^{{\bf k}}=\sum_{n=0}^{d} nm_n = dN-\sum_{n=1}^d nm_{d-n}.$$
By Lemma \ref{LemmaATR7}, we have:
$$n=0, \ldots, d-1,m_{n}\leq \bar k_n p^{e_{\bar k_0+\cdots + \bar k_n}},$$
where we recall that:
$$N=\sum_{n=1}^{\ell_p(N)} p^{e_n}.$$
Let $l= \ell_p(N)-1.$ Observe that:
$$e_{\ell_p(N) -p-(p-1)t}\leq e_{l}-1-t.$$
Since $d(p-1)\leq \ell_p(N)-p,$ we get:
$$n=1, \ldots, d,\bar k_0+\cdots + \bar k_{d-n}\leq (p-1) (d-n+1)\leq \ell_p(N)-p -(p-1)(n-1).$$
Thus:
$$n=1, \ldots, d,m_{d-n}\leq (p-1) p^{e_{l}-n},$$
Therefore:
$$\sum_{n=1}^d nm_{d-n} \leq (p-1)p^{e_{l}} \sum_{n=1}^d np^{-n}.$$
Recall that if $x\in \mathbb R\setminus \{1\},$ we have:
$$\sum_{n=1}^d nx^{n-1}= \frac{1-x^{d+1}+(d+1)(x-1)x^d }{(x-1)^2}.$$
Thus:
$$(p-1)\sum_{n=1}^d np^{-n}= \frac{p-p^{-d}-(d+1)(p-1)p^{-d}}{p-1}<\frac{p}{p-1}. $$
Now:
$$\sum_{n=1}^d nm_{d-n} <\frac{p}{p-1}p^{e_l}.$$
Thus:
$$\sum_{n=1}^d nm_{d-n} < 2p^{e_l}.$$
But $2p^{e_l}\leq N$ since $l=\ell_p(N)-1.$ Thus:
$$\deg_t \sum_{a\in A_{+,d}} a(t)^N a^{{\bf k}}= dN-\sum_{n=1}^d nm_{d-n}> dN -N.$$
\end{proof}

%%%%%%%%%%%%%%%%%%%%%%%%%%%%%%%%%%%%%%%%%%%
\subsection{Newton polygons of truncated $L$-series}${}$\par
 For $i,j\geq 0,$ we set:
$$S_j(i)= \sum_{a\in A_{+,j}}a(t)^i\in \mathbb F_p[t].$$
Note that by Proposition \ref{PropositionATR2}, we have $S_{i}(N)\not =0$ for $i=0, \ldots, r.$
\begin{proposition}\label{PropositionATR4} We have:
$$i=0, \ldots, r, \deg_t \alpha_{i,N}(t) = \deg_t S_{i}(N).$$
\end{proposition}
\begin{proof} We recall that:
$$r={\rm Max} \{ [\frac{\ell_p(N)-p}{p-1}], 0\} \in \mathbb N.$$ We can assume that $r\geq 1.$ By Proposition \ref{PropositionATR3}, for $i=0, \ldots,r, $   
$${\rm Max}\{ \deg_t \sum_{a\in A_{+,\ell({\bf k})}} a(t)^N a^{{\bf k}}, w({\bf k})+ \ell ({\bf k})=i\}$$
is attained for a unique ${\bf k}$ which is $(0, \ldots, 0)\in \mathbb N^{i}.$ It remains to apply Lemma \ref{LemmaATR2}. 
\end{proof}
We set:
$$\Lambda_r(N)= \sum_{i=0}^r\alpha_i(N) \theta^{-i} \in K_\infty[t].$$ 
Let's write:
$$N=\sum_{l=0} ^k n_l p^l, n_0, \ldots, n_k\in \{0, \ldots p-1\}, n_k\not =0.$$
\begin{proposition}\label{PropositionATR5}
We have $\deg_t\Lambda_r(N)= \deg_tS_r(N).$  Furthermore the edge points of the Newton polygon of $\Lambda_r(N)$ are:
$$(\deg_tS_j(N),j), j=0, \ldots, r.$$
\end{proposition}
\begin{proof} We can assume that $r\geq 1.$ By Proposition \ref{PropositionATR2}, we have $U_r(N, (p-1, \ldots, p-1))\not =\emptyset.$ Let ${\bf m}\in \mathbb  N^{r+1}$ be the optimal element of $U_r(N, (p-1, \ldots, p-1))$ given by Proposition \ref{PropositionATR2} and Lemma \ref{LemmaATR7}. For $j=0, \ldots, r,$ let ${\bf m}(j)=(m_0, \ldots, m_{j-1}, N-\sum_{n=0}^{j-1}m_n)\in \mathbb N^{j+1}.$  Then, again by Proposition \ref{PropositionATR2} and Lemma \ref{LemmaATR7},  ${\bf m}(j)$ is the optimal element of $U_{j}(N, (p-1, \ldots, p-1)).$ Therefore, $\deg_tS_j(N)=\deg {\bf m}(j), j=0, \ldots,r.$ For $j=0, \ldots, r,$ we have:
$$p^{k+1}>N-\sum_{n=0}^{j-1}m_n >n_kp^k.$$
Now let $j\in \{ 0, \ldots, r-1\},$ we have:
$$p^{k+1}>\deg {\bf m}(j+1)-\deg{\bf m}(j)= N-\sum_{n=0}^{j}m_n >n_k p^k.$$
Thus, by Proposition \ref{PropositionATR4}, we get $\deg_t\Lambda_r(N)= \deg_tS_r(N).$ Furthermore, we observe that for $j\in \{ 1, \ldots, r-1\},$ we have:
$$\deg {\bf m}(j)-\deg{\bf m}(j-1)>\deg {\bf m}(j+1)-\deg{\bf m}(j).$$
Thus, one easily sees that  the edge points of the Newton polygon of $\Lambda_r(N)$ are $(\deg {\bf m}(j), j), j=0, \ldots, r.$ \end{proof}
%%%%%%%%%%%%%%%%%%%%%%%%%%
\subsection{A positive answer to Problem \ref{Conjecture}}${}$\par
Let $N\geq 2$ be an integer, $N\equiv 1\pmod{p-1}.$ Recall that:
$$N=\sum_{l=0} ^k n_l p^l, n_0, \ldots, n_k\in \{0, \ldots p-1\}, n_k\not =0.$$
Recall that $\ell_p(N)= n_0+\cdots +n_k$ and $r={Max }\{ \frac{\ell_p(N)-p}{p-1}, 0\}.$ Let $b_N\in \mathbb N$ be the total degree in $t, \theta$ of the polynomial $B_N(t).$
\begin{proposition}\label{PropositionATR6}
The polynomial $B_N(t,\theta)$ has only one monomial of total degree $b_N$, which is of the form $ t^{b_N}.$ Furthermore:
$$b_N=\deg_tS_r(N) .$$
\end{proposition}
\begin{proof}  We can assume that $r\geq 1.$  First let's observe that:
$$\prod_{j\geq 1} ( 1-\theta^{1-p^j})^{-1}B_N(t,\theta) \equiv (-1)^{\frac{\ell_p(N)-1}{p-1}}\delta_N \prod_{l=0}^k \prod_{j\geq 0} (1-\frac{t^{p^l}}{\theta^{p^{j}}})^{-n_l}\pmod{\frac{1}{\theta} \mathbb F_p[t][[\frac{1}{\theta}]]},$$
where:
$$\delta_N=\sum_{i=0}^r \alpha_{r-i}(N) \theta^i.$$
Let $\varepsilon _N \in \mathbb F_p[t, \theta]$ be uniquely determined by the congruence:
$$\varepsilon_N\equiv \delta_N \prod_{l=0}^k \prod_{j\geq 0} (1-\frac{t^{p^l}}{\theta^{p^{j}}})^{-n_l}\pmod{\frac{1}{\theta} \mathbb F_p[t][[\frac{1}{\theta}]]}.$$ 
Let $i\in \{ 0, \ldots, r\}.$ A monomial in the product 
$$\alpha_i(N) \theta^i \prod_{l=0}^k\prod_{j\geq 0} (1-\frac{t^{p^l}}{\theta^{p^{j}}})^{-n_l}\pmod{\frac{1}{\theta} \mathbb F_p[t][[\frac{1}{\theta}]]}$$
is of the form:
$$ \theta^i t^j \frac{t^\alpha}{\theta^\beta}, \, j\leq \deg_t\alpha_i(t), \beta\leq i, \alpha \leq p^k\beta.$$
Thus the total degree of a monomial in
$$\alpha_i(N) \theta^i \prod_{l=0}^k\prod_{j\geq 0} (1-\frac{t^{p^l}}{\theta^{p^{j}}})^{-n_l}\pmod{\frac{1}{\theta} \mathbb F_p[t][[\frac{1}{\theta}]]}$$
is less than or equal to $ p^k i+\deg_t\alpha_i(t).$
To conclude the proof of the Proposition, we use the same arguments as that used in the proof of Proposition \ref{PropositionATR5}. By Proposition \ref{PropositionATR2}, we have $U_r(N, (p-1, \ldots, p-1))\not =\emptyset.$ Let ${\bf m}\in \mathbb  N^{r+1}$ be the optimal element of $U_r(N, (p-1, \ldots, p-1))$ given by Proposition \ref{PropositionATR2} and Lemma \ref{LemmaATR7}. For $j=0, \ldots, r,$ let ${\bf m}(j)=(m_0, \ldots, m_{r-j-1},N-\sum_{n=0}^{r-j-1}m_n)\in \mathbb N^{r-j+1}.$ Then, again by Proposition \ref{PropositionATR2} and Lemma \ref{LemmaATR7},  ${\bf m}(j)$ is the optimal element of $U_{r-j}(N, (p-1, \ldots, p-1)).$ For $j=0, \ldots, r,$ we have:
$$N-\sum_{n=0}^{r-j-1}m_n >p^k.$$
Now let $j\in \{ 0, \ldots, r-1\},$ we have:
$$\deg {\bf m}(j)-\deg{\bf m}(j+1)= N-\sum_{n=0}^{r-j-1}m_n > p^k.$$
Thus, by Proposition \ref{PropositionATR3} and Proposition \ref{PropositionATR4},   ${\rm Max}\{p^k i+\deg_t\alpha_i(t), i=0, \ldots, r\}$ is attained  exactly at $i=0.$ Again by Proposition \ref{PropositionATR4}, this implies that the total degree in $t, \theta$ of $\varepsilon_N$ is equal to $\deg_t S_r(N)$ and that $\varepsilon_N(t)$ has only one monomial of total degree $\deg_t S_r(N)$ which is of the form $ t^{\deg_t S_r(N)}.$  The Proposition follows.
\end{proof}
\begin{theorem}\label{TheoremATR1} Let $N\geq 2, N\equiv 1\pmod{p-1}.$ The polynomial $B_N(\theta, t)$ (viewed as a polynomial in $t$) has $r$ simple roots and all its roots are contained in $\mathbb F_p((\frac{1}{\theta}))\setminus\{\theta^{p^i} ,i \in \mathbb Z\}.$
\end{theorem}
\begin{proof}${}$\par
 Recall that $B_N(t,\theta)$ is a monic polynomial in $\theta$ such that $\deg_\theta B_N(t,\theta)=r.$ We can assume that $r\geq 1.$ By Proposition \ref{PropositionATR6}, the leading coefficient of $B_N(t,\theta)$ as a polynomial in $t$ is in $\mathbb F_p^\times$ and:
$$b_N=\deg_tB_N(t,\theta)>r.$$ Let $S=\{ \theta^{p^i}, i\geq -k\}.$  Then if $\alpha \in \mathbb C_\infty$ is a zero of $L_N(t)$ and $\alpha \not \in S,$ $\alpha$ must be a zero of $B_N(t,\theta).$ Observe that by Proposition \ref{PropositionATR5} and Proposition \ref{PropositionATR6}, we have:
$$\deg_t\Lambda_r(N)=\deg_tB_N(t, \theta).$$
By Proposition \ref{PropositionATR5}, the zeros in $\mathbb C_\infty$ of $\Lambda_r(N)$ are not in $S.$ By  Lemma \ref{LemmaATR1} and Lemma \ref{LemmaATR4}, $\theta^{-r}B_N(t, \theta)$ and $\Lambda_r(t)$ have the same Newton polygon. Thus, by the proof of Proposition  \ref{PropositionATR5} and the properties of Newton polygons ( \cite{GOS}, chapter 2), we get in $K_\infty[t]:$
$$B_N(t, \theta) =\lambda \prod_{j=1} ^r P_j(t),$$
where $\lambda\in \mathbb F_p^\times, $ $P_j(t)$ is an irreducible monic element in $K_\infty[t],$ $P_j(t)\not =P_{j'}(t) $ for $j\not =j'.$  Furthermore each root of $P_j(t)$ generates a totally ramified extension of $K_\infty$ and  $p^{k+1} >\deg_tP_j(t) > n_kp^k.$ Also note that 
$\deg_tP_j(t)\not \equiv 0\pmod{p^k} $ and $\deg_t P_j(t)\equiv 1\pmod{p-1}.$ Observe that if $x\in \cup_{i\in \mathbb Z}(\mathbb F_p((\frac{1}{\theta^{p^i}}))^{ab})^{perf},$ then there exist $l\geq 0, m\in \mathbb Z,$ $d\geq 1,$ $p\equiv 1\mod{d},$ such that $v_\infty (x)= \frac{m}{dp^l}.$ Thus, $P_j(t) $ has no roots in $\cup_{i\in \mathbb Z }(\mathbb F_p((\frac{1}{\theta^{p^i}}))^{ab})^{perf}.$ \par
 Write:
$$\theta^{-r}B_N(t, \theta) =\sum_{j=0}^r \beta_j(t) \theta^{-j}, \beta_j(t)\in \mathbb F_p[t].$$
Observe that $\beta_0(t)=1$ and by the above discussion,  $\theta^{-r}B_N(t, \theta)$ and $\Lambda_r(t)$ have the same Newton polygon (as polynomials in $t$). Now, by Proposition \ref{PropositionATR5} and Proposition \ref{PropositionATR6}, we get:
$$\deg_t\beta_r(t)=\deg_t\alpha_r(t).$$
We deduce that:
$$i=0, \ldots,r, \deg_t\beta_i(t)=\deg_t\alpha_i(t).$$
By the proof of  Proposition \ref{PropositionATR5}, for $i\in \{1, \ldots, r-1\},$ $\deg_t\beta_{i+1}(t)-\deg_t \beta_i(t)<\deg_t\beta_{i}(t)-\deg_t \beta_{i-1}(t) .$ Thus the edges of the Newton polygon  of $\theta^{-r} B_N(t, \theta)$ viewed as polynomial in $\frac{1}{\theta}$ are $(i, -\deg_t(\beta_i(t))), i=0, \ldots, r.$ 
\end{proof}
%%%%%%%%%%%%%%%%%%%%%%%%%%%%%%%%%%%%%%%%%%%%%%%%%%%%%%%%%%%%%%%%%%%%%%%%%%%%
\section{Some hints for  Problem \ref{Conjecture} for  general $q.$}
In this section $q$ is no longer assumed to be equal to $p.$\par
%%%%%%%%%%%%%%%%%%%%
\subsection{The work of J. Sheats}${}$\par
For $N,d\geq 1,$ we set $U_d(N)=U_d(N, (q-1, \ldots, q-1)).$  Thus:
$$S_d(N):= \sum_{a\in A_{+,d}} a(t)^N=(-1)^d\sum_{{\bf m} \in U_d(N)}C(N, {\bf m}) t^{\deg {\bf m}} .$$
J. Sheats  proved (\cite{SHE}, Lemma 1.3) that if $U_d(N)\not =\emptyset,$ $U_d(N)$ has a unique optimal element and it is the greedy element of $U_d(N).$  In particular $U_d(N)\not =\emptyset \Leftrightarrow S_d(N)\not =0.$  Observe that if ${\bf m}=(m_0, \ldots, m_{d})\in U_d(N),$ then $(m_0, \ldots, m_{d-2}, m_{d-1}+m_d) \in U_{d-1}(N).$ In particular $U_d(N)\not =\emptyset \Rightarrow U_{d-1}(N)\not =\emptyset.$
\begin{proposition}\label{PropositionATR7}${}$\par
\noindent 1) Let $d\geq 1$ such that $U_d(N)\not =\emptyset.$ Then, for $j\in \{1, \ldots, d-1\},$ we have:
$$\deg_tS_j(N)-\deg_tS_{j-1}(N)>\deg_tS_{j+1}(N)-\deg_tS_j(N).$$
\noindent 2) Let $d\geq 1$ such that $U_{d+1}(N)\not =\emptyset.$ Let ${\bf m}$ be the greedy element of $U_{d+1}(N).$ Then:
$$\deg_tS_d(N)>N(d-1),$$
and:
$$\deg_tS_d(N)-\deg_tS_{d-1}(N) >m_{d+1}.$$
\end{proposition}\par
\begin{proof}${}$\par
\noindent  1) Observe that this assertion  is a consequence of the proof of \cite{SHE}, Theorem 1.1 (see pages 127 and 128 of \cite{SHE}).\par
\noindent 2) Let ${\bf m}=(m_0, \ldots, m_{d+1})$ be the greedy element of $U_{d+1}(N).$ Define ${\bf m}'=(m_0, \ldots, m_{d})\in U_d(N-m_{d+1}).$ Then:
$$m_d\equiv 0\pmod{q-1}, m_d \geq q-1.$$
Furhermore, observe that ${\bf m}'$ is the greedy element of $U_d(N-m_{d+1}).$ By \cite{SHE}, Lemma 1.3 and Proposition 4.6, we get:
$$\deg_tS_d(N-m_{d+1})>(N-m_{d+1}) (d-1).$$
Let ${\bf m}''=(m_0, \ldots, m_{d-1}, m_d+m_{d+1})\in U_d(N).$ We have:
$$\deg_tS_d(N)\geq m_1+\cdots +(d-1) (m_{d-1})+d(m_d+m_{d+1}).$$
Thus:
\begin{eqnarray*}
\deg_tS_d(N)&\geq&  \deg_tS_d(N-m_{d+1})+dm_{d+1}\\
&>& (N-m_{d+1}) (d-1)+dm_{d+1}= (d-1)N +m_{d+1}.
\end{eqnarray*}
Thus:
$$\deg_t S_d(N)>N(d-1),$$
$$\deg_tS_d(N)-\deg_tS_{d-1}(N) \geq \deg_tS_d(N)-(d-1)N>m_{d+1}.$$
\end{proof}
To conclude this paragraph, we recall the following crucial result due to G. B\"ockle (\cite{BOC}, Theorem 1.2):
$$S_d(N)\not =0\Leftrightarrow d(q-1)\leq {\rm Min}\{ \ell_q(p^iN), i\in \mathbb N\}.$$
An integer $N\geq 1$ will be called $q$-minimal if:
$$[\frac{\ell_q(N)}{q-1}]={\rm Min}\{ [\frac{\ell_q(p^iN)}{q-1}], i\in \mathbb  N\}.$$
%%%%%%%%%%%%%%%%
\subsection{Consequences of Sheats results}${}$\par
Let $N\geq 1,$ and write:
$$L_N(t)=\sum_{i\geq 0} \alpha_{i,N}(t) \theta^{-i}, \alpha_{i,N}(t) \in \mathbb F_q[t].$$
\begin{proposition}\label{PropositionATR8} Let $d\geq 1$ such that $U_{d+1}(N)\not =\emptyset.$ Set: 
$$\Lambda_d(t)= \sum_{i=0}^d\alpha_{i,N}(t) \theta^{-i}\in K_\infty[t].$$
Then $\deg_t\Lambda_d(t)=\deg_tS_d(N)$ and the edge points of the Newton polygon of $\Lambda_d(t)$ are:
$$(\deg_tS_i(N),i), i=0, \ldots, d.$$
\end{proposition}
\begin{proof} The proof uses  similar arguments as that used in the proof of Proposition \ref{PropositionATR5}. Let $j\geq 0,$ then (see Lemma \ref{LemmaATR2}), we have:
$$\alpha_{j,N}(t)=\sum_{\ell({\bf k})+w({\bf k})=j} (-1)^{\ell ({\bf k})}C_{{\bf k}}\sum_{a\in A_{+,\ell({\bf k})}} a(t)^N a^{{\bf k}}.$$
Observe that:
$$\deg_t \sum_{a\in A_{+,\ell({\bf k})}} a(t)^N a^{{\bf k}}\leq \ell({\bf k})N.$$
Thus, for $j=0, \ldots, d,$ by Proposition \ref{PropositionATR7}, assertion 2),  we get:
$$\deg_t\alpha_{j,N}(t)= \deg_tS_j(N).$$
In particular, again by Proposition \ref{PropositionATR7}, assertion 2), we have: 
$$\deg_t\Lambda_d(t)=\deg_tS_d(N).$$
Finally, by Proposition \ref{PropositionATR7}, assertion 1), $(\deg_tS_i(N),i), i=0, \ldots, d,$ are the edge points of the Newton polygon of $\Lambda_d(t).$

\end{proof}
\begin{lemma}\label{LemmaATR8} We assume that $N$ is $q$-minimal, $N\equiv 1\pmod{q-1}.$  We also assume that $r\geq 1$ (recall that $r=\frac{\ell_q(N)-q}{q-1}$).\par
\noindent 1) Write $N=\sum_{l=0}^k n_lq^l, n_0, \cdots ,n_k \in \{0, \ldots, q-1\}, n_k\not =0.$ Then there exists an integer $0\leq m\leq k$ such that $n_m\not \equiv 0\pmod{p}.$\par
\noindent 2) Let $n={\rm Max}\{ l, 0\leq l\leq k, n_l\not \equiv 0\pmod{p}\}.$ Let ${\bf m}=(m_0, \ldots, m_{r+1})\in U_{r+1}(N)$ be the greedy element, then:
$$m_{r+1} =q^n.$$
\noindent 3) Assume that $n<k.$  Then, for $i\in \{ 1, \ldots, r-1\},$ we have:
$$\deg_tS_i(N)-\deg_tS_{i-1}(N)>pq^k.$$
\end{lemma}
\begin{proof}${}$\par
\noindent 1) Let's assume that the assertion is false. Then:
$$\ell_q(\frac{q}{p}N)=\frac{\ell_q(N)}{p}.$$
This contradicts the $q$-minimality of $N.$\par
\noindent 2) Observe that:
$$\forall i\geq 0, \ell_q(p^iN)= \ell_q(p^iN-p^iq^n)+\ell_q(p^i).$$
Therefore, $N-q^n$ is $q$-minimal. Thus, by B\"ockle's result:  $U_{r+1}(N-q^n )\not = \emptyset.$ This easily implies that there exits ${\bf n}=(n_0, \ldots, n_{r+1})\in U_{r+1}(N)$ such that:
$$n_{r+1}=q^n.$$
We have:
$$1+(q-1)(r+1)=\ell_q(N)= \sum_{i=0}^{r+1}\ell_q(m_i).$$
Furthermore:
$$\sum_{i=0}^r\ell_q(m_i)\equiv 0\pmod{q-1},$$
and
$$ \sum_{i=0}^r\ell_q(m_i) \geq (q-1) (r+1).$$
Thus:
$$\ell_q(m_{r+1})=1.$$
This implies that $m_{r+1}$ is a power of $q$ and since ${\bf m}$ is the greedy element of $U_{r+1}(N),$ we also have: $m_{r+1}\mid\geq q^n.$ Since  there is no carryover $p$-digits in the sum $m_0+\cdots+m_{r+1},$ by the definition of $n,$ we deduce that $ m_{r+1}=q^n.$\par

\noindent 3) Let ${\bf m}'=(m_0, \ldots, m_{r-1}, m_r+m_{r+1})\in U_r(N).$  If ${\bf n}$ is the greedy element of $U_r(N)$ then:
$$n_r\geq m_r+m_{r+1}.$$
Since ${\bf m}$ is the greedy element of $U_{r+1} (N), $ we have:
$$m_0\leq m_1\leq \cdots \leq m_r.$$
Since  there is no carryover $p$-digits in the sum $m_0+\cdots+m_{r+1},$ and $n<k, $ this implies that:
$$m_r=\sum_{l=0}^k j_l q^l, j_l\in \{ 0, \ldots, q-1\}, j_k\not =0, j_k\equiv 0\pmod{p}.$$
Thus:
$$m_r>pq^k.$$
It remains to apply Proposition \ref{PropositionATR7}.
\end{proof}

%%%%%%%%%%%%%%%%
\subsection{Zeros of $B_N(\theta,t)$}${}$\par
The following theorem implies in particular Theorem \ref{TheoremATR-BC} in the case where $N$ is $q$-minimal.
\begin{theorem}\label{TheoremATR2} We assume that $N$ is $q$-minimal, $N\equiv 1\pmod{q-1}.$  We also assume that $r\geq 1.$\par
\noindent 1) $B_N(\theta, \theta)\not =0$ and the zeros of $B_N(t, \theta)$ are algebraic integers (i.e. they are integral over $A$). Furthermore:  $(r-1)N<\deg_tB_N(t, \theta) <rN.$\par
\noindent 2) $B_N(\theta, t)$ has only simple roots and its roots belong to $\mathbb F_p((\frac{1}{\theta}))\setminus\{\theta\}.$
\end{theorem}
\begin{proof} The proof is in the same spirit as that of the proofs of  Proposition \ref{PropositionATR6} and  Theorem \ref{TheoremATR1}.\par
 Write $N=\sum_{l=0}^k n_l q^l,$ where  $n_0, \ldots, n_k\in \{0, \ldots, q-1\},$ and  $n_k \not =0.$ Recall that $r+1=[\frac{\ell_q(N)}{q-1}].$ Let $n\geq 0$ be the integer such that $n={\rm  Max}\{l, n_l\not \equiv 0\pmod{p}\}$  (see Lemma \ref{LemmaATR8}). Let $\varepsilon_N(t)\in \mathbb F_p[t, \frac{1}{\theta}]$ be the polynomial determined by the congruence:
$$\varepsilon_N(t) \equiv  \Lambda_r(t) \prod_{j=0}^k\prod_{i\geq 0}(1-\frac{t^{q^j}}{\theta^{q^i}})^{n_j}\pmod{\frac{1}{\theta^{r+1}}\mathbb F_p[t][[\frac{1}{\theta}]]},$$
where $\Lambda_r(t)=\sum_{l=0}^r \alpha_{l,N}(t) \theta^{-l}.$ We can write:
$$\varepsilon_N(t) =\sum_{l=0}^r \eta_l(t) \theta^{-l}, \eta_l(t)\in \mathbb F_p[t].$$
Note that $\eta_l(t)$ is a $\mathbb F_p$-linear combination   of terms of the form 
$$x_{l,j,u}= \alpha_{l-u,N}(t) \theta^{-l+u} t^j \theta^{-u}, \ j\leq q^k u.$$
By Proposition \ref{PropositionATR8}, we have:
$$l=0, \ldots, r, \, \deg_t\alpha_{l, N}(t)=\deg_tS_l(N).$$
\noindent 1) Case $n=k.$\par
\noindent By Proposition \ref{PropositionATR7} and Lemma \ref{LemmaATR8} :
$$l=0, \ldots , h-1, \, \deg_t\alpha_{h-l, N}(t)< \deg_t\alpha_{h,N}(t)-q^k l.$$
Thus, if $l\not =h$ or $u\not =0,$   we get:
$$l=0, \ldots, h, \deg_tx_{l,j,u}<\deg_t\alpha_{l,N}(t).$$
Therefore:
$$l=0, \ldots, r, \, \deg_t\eta_l(t)= \deg_tS_l(N).$$
2) Case $n<k.$\par
\noindent  As in the proof of the case $n=k,$ we get by Proposition \ref{PropositionATR7} and Lemma \ref{LemmaATR8}:
$$l=0, \ldots, r-1, \, \deg_t\eta_l(t)= \deg_tS_l(N).$$
Furthermore, by Proposition \ref{PropositionATR7} and Lemma \ref{LemmaATR8}, we have for $l\geq 2$:
$$\deg_tS_{r-l}(N)< \deg_tS_{r-1}(N)-pq^k(l-1),$$
$$\deg_tS_r(N)-\deg_tS_{r-1}(N)>q^n.$$
Thus, for $u\geq 2$:
$$\deg_tx_{r,j,u}\leq \deg_tS_{r-u}(N)+q^ku< \deg_tS_r(N).$$
Now, observe that:
$$\prod_{j=0}^k\prod_{i\geq 0}(1-\frac{t^{q^j}}{\theta^{q^i}})^{n_j}\equiv 1-\frac{1}{\theta}\sum_{l=0}^k n_l t^{q^l} \pmod{\frac{1}{\theta^2}\mathbb F_p[t][[\frac{1}{\theta}]]}.$$
Thus:
$$\deg_tx_{r,j,1}\leq \deg_tS_{r-1}(N)+q^n <\deg_tS_r(N).$$
Thus we get:
$$\deg_t\eta_r(t)=\deg_tS_r(N).$$
Now, observe that:
$$\frac{(-1)^{\frac{\ell_q(N)-1}{q-1}}\theta^{-r}B_N(t, \theta)}{\prod_{j\geq 1} (1- \theta^{1-q^j})} \equiv \varepsilon_N\pmod{\frac {1}{\theta^{r+1}}\mathbb F_p[t][[\frac{1}{\theta}]]}.$$
We easily deduce that:
$$\theta^{-r}B_N(t, \theta)=\sum_{l=0}^r\beta_l(t) \theta^{-l} , \beta_l(t)\in \mathbb F_p[t], \deg_t\beta_l(t)=\deg_tS_l(N), l=0, \ldots, r.$$
Observe that, by Proposition \ref{PropositionATR7}, we have $\deg_tS_r(N)>N(r-1),$ and it is obvious that $\deg_tS_r(N)<rN.$  Now, we get assertion 1) and 2) by the same reasoning as that used in the proof of Theorem \ref{TheoremATR1}.
\end{proof}
\begin{corollary}\label{CorollaryATR1}
We assume that $N$ is $q$-minimal, $N\equiv 1\pmod{q-1}.$  We also assume that $r\geq 1.$ Then, $B_N(t, \theta)$ has at most one zero in $\{ \theta^{q^i}, i\in \mathbb Z\}.$
\end{corollary}
\begin{proof} Let's assume that $B_N(t, \theta)$ has a zero $\alpha\in \{ \theta^{q^i}, i\in \mathbb Z\}.$ Let $n$ be the integer introduced in the proof of Theorem \ref{TheoremATR2}. Then:
$$\alpha=\theta^{q^{-i}} , k\geq  i>n,$$
where $q^k\leq N< q^{k+1}.$
 Thus $n<k,$ and therefore, by Lemma \ref{LemmaATR8}, we must have:
 $$\deg_tS_r(N)-\deg_tS_{r-1}(N)= q^i.$$
 By the proof of Theorem \ref{TheoremATR2}, we have:
 $$\theta^{-r}B_N(t,\theta)= \left(\frac{t^{q^i}}{\theta}-1\right)F(t),$$
 where $F(t)=\sum_{l=0}^{r-1} \nu_l(t) \theta^{-l}, \nu_l(t)\in \mathbb F_p[t], \deg_t\nu_l(t)=\deg_tS_l(N), l=0, \ldots , r-1.$ Furthermore $\nu_0(t)=-1.$This implies that  the zeros of $F(t)$ are not in $\{ \theta^{q^i}, i\in \mathbb Z\}.$  \end{proof}
  
\begin{corollary}\label{CorollaryATR2}${}$\par
\noindent 1) The polynomial $\mathbb  B_s$ is square-free, i.e. $\mathbb B_s$ is not divisible by the square of a non-trivial polynomial in $\mathbb F_q[t_1, \ldots, t_s, \theta].$\par
\noindent 2) For all $l, n\in \mathbb N,$ $\mathbb B_s$ is relatively prime to $(t_1^{q^{l}}-\theta^{q^{n}})$ (observe that $\mathbb B_s$ is a symmetric polynomial in  $t_1, \ldots, t_s$). \par
\noindent 3) For all monic irreducible prime $P$ of $A,$ $\mathbb B_s$ is relatively prime to $P(t_1)\cdots P(t_s)-P.$
\end{corollary} 
\begin{proof} Let $N=q^{e_1}+\ldots +q^{e_s},$ $0\leq e_1<e_2<\ldots < e_s.$ Then:
$$B_N(t, \theta)=\mathbb B_s\mid_{t_i=t^{q^{e_i}}}.$$
We observe that $N$ is $q$-minimal. Thus we can apply Theorem \ref{TheoremATR2}. This Theorem and its proof imply that $B_N(t, \theta)$ is square-free and has no roots in $\{ \theta^{q^i}, i\in \mathbb Z\}.$ This proves 1) and 2).\par
\noindent Let $P$ be a monic irreducible polynomial in $A.$ Suppose that that $P(t_1)\cdots P(t_s)-P$ and $\mathbb B_s$ are not relatively prime.Then $P(t)^N-P$ and $B_N(t, \theta)$ are not relatively prime. But, by the proof of Theorem \ref{TheoremATR2}, if $\alpha \in \mathbb C_\infty$ is a root  of $B_N(t, \theta),$ then:
$$v_\infty(\alpha)>\frac{-1}{N}.$$
Now, observe that if $\beta\in \mathbb C_\infty$ is a root of $P(t)^N-P,$ then $v_\infty(\beta)=\frac{-1}{N}.$ This leads to a contradiction.
\end{proof}
Note that assertion 1) of the above Corollary gives the cyclicity result implied by    \cite{ATR}, Theorem 3.17, but  by a completely different method.
%%%%%%%%%%%%%%%%%%%%%%%%%%%%%%%%%%%%%%%%%%%%%%%%%%%%%%%%%%%%%%
\section{An example}
We study here an example of an $N$ which is not $q$-minimal, so that our method does not apply. We choose $q=4$, and $N= 682 = 2 + 2\times 4 + 2\times4^2+2\times 4^3 + 2\times 4^4$. We get $l_q(N)=10 = 3q-2$ so that $\deg_\theta(B_N(t, \theta))=2$. Moreover, $l_q(pN)=5$ so that $N$ is not $q$-minimal.
By using the table of section \ref{tables}, we get : 
\begin{align*}
B_N(t, \theta) =  \theta^2 & +\theta \left( t^{10} + t^{34} +t^{40} +t^{130} + t^{136} + t^{160} + t^{514} + t^{520}+t^{544} + t^{640}\right) \\
& + \left(t^{170} +t^{554} + t^{650} + t^{674} + t^{680}\right).
\end{align*}
The Newton polygon of $B_N(t, \theta)$ has then the end points $(0,-2), (640,-1), (680, 0)$. We deduce that $B_N(t, \theta)$ has $640$ distinct zeroes of valuation $-\frac 1 {640}$ and $40$ distinct zeros of valuation $-\frac 1 {40}$.
Similarly, $B_N(\theta, t)$ has two zeros of respective valuations $-40$ and $-640$.
In particular, we still have that $B_N(t,\theta)$ has no zero of the form $\theta^{q^i}, i\in \mathbb Z$, and an affirmative answer to Problem \ref{Conjecture}.

\renewcommand\thesection{\Alph{section}}

\renewcommand{\thesection}{A}

\section{Appendix: The digit principle and derivatives of certain $L$-series, by B. Angl\`es, D. Goss, F. Pellarin \and F. Tavares Ribeiro}

We keep the notation of the article.

Let $N$ be a positive integer. We consider its base-$q$ expansion $N=\sum_{i=0}^r n_i q^i,$ with $n_i\in \{ 0, \ldots, q-1\}.$ We recall that $\ell_q(N)=\sum_{i=0}^r n_i$ and the definition of the Carlitz factorial :
$$\Pi(N)= \prod_{i\geq 0} D_i^{n_i}\in A^+,$$
where  $[k]=\theta^{q^k}-\theta$ if $k>0$ and $D_j=[j][j-1]^q\cdots[1]^{q^{j-1}}$ for $j>0$.

It is easy to see (the details are in \S \ref{andersonthakur}, \ref{onedigit} and 
\ref{severaldigits}) that,
if we denote by $a'$ the derivative $\frac{d}{d\theta}a$ of $a\in A$ with respect to $\theta$,
the series $$\sum_{k\geq 1}\sum_{a\in A_{+,k}}\frac{a'^N}{a}$$ 
converges in $K_{\infty}$ to a limit that we denote by $\delta_N$. 
In particular, if $n=q^j$ with $j>0$, we will see
(Proposition \ref{proposition2}) that 
$$\delta_1=-\sum_{k\geq1} \frac{1}{[k]} \ \text{ and } \delta_{q^j}=\frac{D_j}{[j]}\widetilde{\pi}^{1-q^j}.$$ 
Our aim  is to prove the following:
\begin{theorem}\label{theorem1}
If $N\geq q$ is such that $N\equiv 1\pmod{q-1}$ and $\ell_q(N)\geq q$, then
$$\frac{\delta_N}{\widetilde{\pi}}=\beta_N\frac{\Pi(N)}{{\Pi([\frac{N}{q}])}^q}\prod_{k= 1}^r\left(\frac{\delta_{q^k}}{\widetilde{\pi}}\right)^{n_k},$$
where for $x\in \mathbb R, $ $[x]$ denotes the integer part of $x,$ and where $\beta_N=(-1)^{\frac{s-1}{q-1}}B_N(\theta, \theta).$
\end{theorem}
Our Theorem \ref{theorem1} can be viewed as a 
kind of {\em digit principle} for the values $\delta_j$ in the sense of  \cite{CON}.

The plan of this appendix is the following. In \S \ref{andersonthakur}, we 
recall the first properties of Anderson and Thakur function $\omega$. In 
\S \ref{onedigit} we discuss the one-digit case of our Theorem, while the general case is discussed in \S \ref{severaldigits}. 

\subsection{The Anderson-Thakur  function}\label{andersonthakur}
Recall that $\bT_t$ denotes the Tate algebra over $\mathbb C_{\infty}$ in the variable $t$, $C: A\rightarrow A\{ \tau\}$ is the Carlitz module (\cite{GOS}, chapter 3), in other words, $C$ is the morphism of $\mathbb F_q$-algebras given by $C_{\theta}=\tau +\theta,$
and 
$$\exp_C=\sum_{i\geq 0}\frac{1}{D_i} \tau ^i \in \mathbb T_t\{ \{ \tau \}\}$$
is the Carlitz exponential. In particular, we have the following equality in $\mathbb T_t\{\{ \tau \}\}:$
$$\exp_C \theta= C_{\theta} \exp_C.$$
Let us choose a $(q-1)$-th root $\sqrt[q-1]{-\theta}$ of $-\theta$ in $\mathbb C_{\infty}$ and set:
$$\widetilde{\pi}=\theta\sqrt[q-1]{-\theta}  \prod_{j\geq 1} (1- \theta^{1-q^j})^{-1} \in \mathbb C_{\infty}^\times.$$
We recall the Anderson-Thakur function (\cite{AND&THA}, proof of Lemma 2.5.4):
$$\omega(t)= \sqrt[q-1]{-\theta}\prod_{j\geq 0}\left(1- \frac{t}{\theta^{q^j}}\right)^{-1}\in \mathbb T_t^\times.$$
To give an idea of how to compute $\exp_C(f)$ for certain $f$ in $\bT_t$,
we verify here that $$\exp_C\left(\frac{\widetilde{\pi}}{\theta-t}\right)=\sum_{j\geq 0} \frac{ \widetilde{\pi}^{q^j}}{D_j(\theta^{q^j}-t)}$$ is a well defined element 
of $\bT_t$. Indeed, for $j\geq 0:$
$$v_{\infty}\left(\frac{ \widetilde{\pi}^{q^j}}{D_j(\theta^{q^j}-t)}\right) = q^j \left(j+1-\frac{q}{q-1}\right).$$
Therefore $\sum_{j\geq 0} \frac{ \widetilde{\pi}^{q^j}}{D_j(\theta^{q^j}-t)}$ converges in $\mathbb T_t.$

We will need the following crucial result in the sequel:
\begin{proposition}\label{proposition1} We have the following equality in $\mathbb T_t:$
$$\omega(t)=\exp_C\left(\frac{\widetilde{\pi}}{\theta-t}\right).$$
\end{proposition}
\begin{proof} It is a consequence of the formulas established in \cite{PEL}. We give details for the convenience of the reader. Let us set $$F(t)=\exp_C\left(\frac{\widetilde{\pi}}{\theta-t}\right).$$ We observe that:
\begin{eqnarray*}
C_{\theta}(F(t))&=&\exp_C\left(\frac{\theta \widetilde{\pi}}{\theta-t}\right) =\exp_C\left(\frac{(\theta-t+t) \widetilde{\pi}}{\theta-t}\right)\\
&=&\exp_C\left(\widetilde{\pi}\right)+\exp_C\left(\frac{t \widetilde{\pi}}{\theta-t}\right)=t\exp_C\left(\frac{ \widetilde{\pi}}{\theta-t}\right)=tF(t).
\end{eqnarray*}
Therefore:
$$\tau (F(t))=(t-\theta) F(t).$$
But we also have:
$$\tau (\omega(t))=(t-\theta)\omega(t).$$
Finally, we get:
$$\tau \left(\frac{F(t)}{\omega(t)}\right)= \frac{F(t)}{\omega(t)}.$$
It is a simple and well-known consequence of a ultrametric variant of Weierstrass preparation Theorem that $\{ f\in \mathbb T_t, \tau (f)=f\} =\mathbb F_q[t].$ Since $\omega\in\bT_t^\times$, we have then:
$$\frac{F(t)}{\omega(t)}\in \mathbb F_q[t].$$
Now observe that
$$F(t)=\exp_C\left(\sum_{j\geq 0}\frac{\widetilde{\pi}}{\theta^{j+1}}t^j\right)=\sum_{j\geq0}\lambda_{\theta^{j+1}}t^j$$
and that, for all $j\geq 0$, $v_{\infty} (\lambda_{\theta^{j+1}})= j+1-\frac{q}{q-1}.$
This implies $v_{\infty}(\frac{F(t)}{\lambda_{\theta}}-1) >0.$
By the definition of $\omega(t),$ we also have
$v_{\infty} (\frac{\omega(t)}{\lambda_{\theta}}-1)>0.$
Thus:
$$v_{\infty} \left(\frac{F(t)}{\omega(t)}-1\right)>0.$$
Since $\frac{F(t)}{\omega(t)}\in \mathbb F_q[t],$ we get
$\omega(t)=F(t).$
\end{proof}
Notice that $\omega(t)$ defines a meromorphic function on $\mathbb C_{\infty}$ without zeroes. Its only poles, simple, are located at
$t=\theta,\theta^q,\theta^{q^2},\ldots$. As a consequence of Proposition \ref{proposition1}, we get:
\begin{corollary}\label{corollary1}
Let $j\geq 0$ be an integer, then:
$$(t-\theta^{q^j})\omega(t)\mid_{t=\theta^{q^j}}= -\frac{\widetilde{\pi}^{q^j}}{D_j}.$$
\end{corollary}
\subsection{The one digit case}\label{onedigit}
Let us consider the following $L$-series:
$$L(t)=L_1(t)=\sum_{k\geq 0}\sum_{a\in A_{+,k}}\frac{a(t)}{a} \in \mathbb T_t.$$
By Proposition \ref{PropositionEX1}, we have the following equality in $\mathbb T_t$ (see \cite{PEL}, Theorem 1):
$$\frac{L(t)\omega(t)}{\widetilde{\pi}}=\frac{1}{\theta-t}.$$
This implies that $L(t)$ extends to an entire function on $\mathbb C_{\infty}$ (see also Lemma \ref{LemmaATR1} or \cite[Proposition 6]{AP}). We set:
$$L'(t)=\sum_{k\geq 0}\sum_{a\in A_{+,k}}\frac{a'(t)}{a} \in \mathbb T_t,$$
where $a'(t)$ denotes the derivative $\frac{d}{dt}a(t)$ of $a(t)$ with respect to $t.$
The derivative $\frac{d}{dt}$ inducing a continuous endomorphism of the 
algebra of entire functions $\bC_\infty\rightarrow\bC_\infty$,
$L'(t)$ extends to an entire function on $\mathbb C_{\infty}.$ Thus, for $j\geq 0$ an integer, $\sum_{k\geq 1}\sum_{a\in A_{+,k}}\frac{a'^{q^j}}{a}$ converges in $K_{\infty}$ and we have:
$$\delta_{q^j}= \sum_{k\geq 1}\sum_{a\in A_{+,k}}\frac{a'^{q^j}}{a}=L'(t)\mid_{t=\theta^{q^j}}.$$
\begin{proposition}\label{proposition2} The following properties hold:
\begin{enumerate}
\item We have:
$$\delta_1=-\sum_{k\geq1} \frac{1}{[k]}.$$
\item Let $j\geq 1$ be an integer, then:
$$\delta_{q^j}=\frac{\Pi(q^j)}{[j]}\widetilde{\pi}^{1-q^j}.$$
\end{enumerate}
\end{proposition}
\begin{proof}
\begin{enumerate}
\item It is well known that, for $n>0$,
$D_n=\prod_{a\in A_{+,n}}a$. Therefore, 
$\sum_{a\in A_{+,n}}\frac{a'}{a}=-\frac{1}{[n]}$ from which the first formula 
follows.
\item By \cite[Remark 8.13.10]{GOS}, we have:
$$L(t)\mid_{t=\theta^{q^j}}=0.$$
Thus:
$$\delta_{q^j}=L'(t)\mid_{t=\theta^{q^j}}= \frac{L(t)}{t-\theta^{q^j}}\mid_{t=\theta^{q^j}}.$$
But, 
$$\frac{\frac{L(t)}{t-\theta^{q^j}}(t-\theta^{q^j})\omega(t)}{\widetilde{\pi}}=\frac{1}{\theta-t}.$$
It remains to apply Corollary \ref{corollary1}.
\end{enumerate}
\end{proof}
\begin{remark}
{\em The transcendence over $K$ of the ``bracket series" $\delta_1=\sum_{i\geq 1}\frac{1}{[i]}$
was first obtained by Wade \cite{WAD}. The transcendence of $\delta_1$ directly implies the transcendence of $\widetilde{\pi}$.}
\end{remark}
\subsection{The several digits case}\label{severaldigits}
As a consequence of  \cite{APT}, Lemma 7.6 (see also \cite{AP}, Corollary 21), the series $L_N(t) = \sum_{d\geq0} \sum_{a\in A_{+,d}} \frac {a(t)^N}a$ has a zero of order at least $N$ at $t=\theta$. Thus,
$$\widetilde L_N(t) = \sum_{d\geq1} \sum_{a\in A_{+,d}} \frac {a'(t)^N}a$$
defines an entire function on $\bC_\infty$ such that
$$\delta_N = \widetilde L_N(\theta).$$

\begin{proof}[Proof of Theorem \ref{theorem1}]
Recall that  $N=\sum_{i=0}^r n_i q^i,$ is the $q$-expansion of $N$.
We set $s=\ell_q(N)$. We can assume that $s\geq q$ by Proposition  \ref{proposition2}. 
From the definition of $B_N(t,\theta)$ in \S\ref{BN}, we have :
$$(-1)^{\frac{s-1}{q-1}}B_N(t,\theta)=L_N(t)\left(\prod_{i=0}^r\omega(t^{q^i})^{n_i}\right)\widetilde{\pi}^{-1}\in A[t].$$ 
Since
$$\delta_N = \widetilde L_N(\theta) = \left(\frac {L_N(t)}{\prod_{i=0}^r(t-\theta^{q^i})^{n_i}} \right)_{\mid t=\theta},$$
we obtain, by Corollary \ref{corollary1} and our previous discussions:
$$\beta_N= \frac{\delta_n\prod_{i= 0}^r(\frac{-\widetilde{\pi}^{q^i}}{D_i})^{n_i}}{\widetilde {\pi}}.$$
Now, by Proposition \ref{proposition2}, we have, for all $k\geq 1$, $D_k=[k]\delta_{q^k}\widetilde{\pi}^{q^k-1}.$
We obtain the Theorem by using the fact that:
$$\frac{\Pi(N)}{\Pi([\frac{N}{q}])^q}=\prod_{k\geq 1} [k]^{n_k}.$$
\end{proof}

\renewcommand{\thesection}{B}

\section{Table}\label{tables}
We give an explicit expression of the polynomials $\mathbb{B}_s$ for $s\in\{q, 2q-1, 3q-2\}$. We recall that $\mathbb{B}_s$ is monic of degree $r=\frac{s-q}{q-1}$.
One obtains the corresponding expressions for $B_N(t, \theta)$ if $\ell_q(N)=s$ by evaluating the variables $t_i$'s as in \S\ref{BN}.
\begin{eqnarray*}
\BB_q &=& 1,\\
\BB_{2q-1} &=& \theta - \sum_{i_1<\dots<i_{q}}t_{i_1}\cdots t_{i_q,}\\
\BB_{3q-2} &=&\theta^2 - \theta \left(\sum_{i_1<\dots<i_{2q-1}} \prod_{j=1}^{2q-1}t_{i_j} + \sum_{i_1<\dots<i_{q}} \prod_{j=1}^{q}t_{i_j} \right) + \\
 &&+ \left(\sum_{i_1<\dots<i_{q}} \prod_{j=1}^{q}t_{i_j}^2\sum_{m_1<\dots<m_{q-1}, m_j\neq i_{j'}} \prod_{j=1}^{q-1}t_{i_j} +  \sum_{i_1<\dots<i_{2q}} \prod_{j=1}^{2q}t_{i_j}\right).
\end{eqnarray*}
One easily computes the discriminant of $\BB_{3q-2}$ from this table. It is then an easy computation to prove that $B_N(\theta,t)$ has only simple roots for all $N$ such that $\ell_q(N)=3q-2$.

%%%%%%%%%%%%%%%%%%%%%%%%%%%%%%%%%%%%%%%%%%%%%%%%%%%%%%%%%%%%%%%%%%%%%%%%%%%%%%%%%%%%%%%%%%%%%%%%%%%%%%%%%%%%%%%%%%%%%%%%%%%%%%%%%%%%%%%%

  \end{document}